\documentclass{Styles_Marcel_Padilla_Bachelor_Thesis}

\Title{Zeros of Random Sections on Line Bundles}
\Thesis{Bachelor Thesis} 
\Author{Marcel Padilla} 
\MNumber{351206} 
\Supervisor{Prof. Dr. Ulrich Pinkall} 
\Referee{Prof. Dr. Boris Springborn}
\Date{18.07.2016} 
\Language{english} 

\newcommand{\R}{\mathbb{R}}

\counterwithout{footnote}{chapter}

\begin{document}

\chapter{Abstracts}

\section{English}

\quad Sections of line bundles on 2 dimensional surfaces in 3 dimensional space can have many distinct shapes. For practical purposes we prefer smooth sections that are visibly easy to follow. This is why smoothing operators have been developed on discrete surfaces as in the inspirational paper \cite{Knoppel:2013:GOD} that can be applied to a any section to return another \emph{smoother} section. We are interested to make predictions on one aspect of the resulting smoothed section's structure, namely position of its signed zeros. The zeros are the most noticeable feature of a section where the section values circles around a specific point.
The purpose of this thesis is to predict the distribution of the smoothed section’s signed zeros with multiplicity that are given by applying the smoothing operator to randomly generated sections of hermitian line bundles on closed simplicial complexes. This will be done in a discrete setting consequently meaning that we will compute the expected sum of indices on each face. Why and how we do this is this thesis' purpose to explain.

\section{Deutsch}

\quad Schnitte von Linienb\"undel auf 2 dimensionalen Oberflaechen in 3 dimensionalem Raum koennen viele Formen annehmen. F\"ur diverse praktische zwecke bevorzugen wir moeglichst glatte Schnitte die angenehmer f\"ur die Augen sind. Genau deswegen wurden Gl\"attungsoperatoren f\"ur discrete Oberfl\"achen entwickelt wie im insperierenden Paper \cite{Knoppel:2013:GOD} die f\"ur jeden Schnitt einen neuen \emph{glatteren} Schnitt zur\"uck geben. Wir wollen eine entscheidene Vorhersage \"uber die Struktur von gegl\"atteten Schnitten machen, n\"amlich \"uber die Verteilung der Nullstellen mit Vorzeichen und Vielfachheit. Diese stellen die offensichtlichsten visuallen merkmalle von Richtungsfeldern da so wie sie in der Anwendung vorkommen. 
Der Zweck dieser Arbeit ist es eine Vorhersage \"uber die Verteilung der Nullstellen mit Vorzeichen von gegl\"atteten Schnitten von hermitischen Linienb\"undel auf Simplizialkomplexen zu machen. Da wir alles in einem discreten Setup durchf\"uhren werden l\"auft es darauf hinaus das wir f\"ur jedes Dreieck die erwartete Summe der Indices der Nullstellen berechnen müssen. Wie und warum wir das machen werden ist das Ziel dieser Arbeit zu beantworten.

\chapter{Introduction}

\section{Acknowledgments}

\quad The main ideas and foundations of this thesis are all collected from the blog posts on the private wordpress website named \emph{Discrete Spin Geometry} on \url{http://brickisland.net/dsg/}. Most of the results from the blog that will be relevant here where archived by Ulrich Pinkall and Felix Kn\"oppel\footnote{both from the Arbeitsgruppe Geometrie, Technische Universität Berlin, Institut für Mathematik}. This thesis collects many ideas from the posts to create a detailed and complete description of the results to be read beyond the bounds of the blogs that are closed to the public. I thank them for their strong support in helping me overcome many of the struggles I had faced inside this thesis.

\section{The Aim of this Thesis}

\quad In this paper we will explore the distribution of the zeros counted with sign of smoothed sections on 2 dimensional discrete surfaces in $\mathbb{R}^3$. We apply the smoothing by the smoothing operator $\mbox{exp}(t\Delta), \ $($t \in \mathbb{R}, \ \Delta$ = Laplace operator) on random sections and aim to compute the expected signed sum of zeros.

In the introduction we will cover up some basic definitions on the underlying structure that we build our statements upon. Then we will have a close look at all the properties that the Laplace operator $\Delta$ has in order to justify why it can be used for sections smoothing (chap: \ref{chapter:Laplace Operator}). After that we examine how we can compute the desired density as we seek and justify the resulting formulas in detail with two different approaches: once with discrete computations (chap: \ref{chap:Discrete Index Approach}) and once using integral geometry (chap: \ref{chapter:Integral Geometry Approach}). At the very end we will combine the two approaches to archive a solid answer on how to compute what we looked for while visualizing the result (chap: \ref{chap:Solution and Discussion}).

\section{Simplicial Complexes and Hermitian Line Bundles}

\quad We will focus on 2 dimensional, closed and connected simplicial complexes $M$ in the 3 dimensional euclidean space $ \R^{3} $. Throughout this thesis we will remain discrete with $ n  \in \mathbb{N} $ vertex points $\{1,...,n\}$ with positions $ v_1 , ... , v_n \in \R^3 $ making up the surface.

\begin{definition}\label{Simplicial Complexes}
Let $V \subseteq \mathbb{R}^{3}$  a  set  of $n  \in \mathbb{N}$ vertices and $\Sigma \subseteq \mbox{P}(V)$ (the power set $P(V)$). Let $\#S :=$ number of elements in the set $S$. Then $(V, \Sigma )$ is called a 2 dimensional simplicial complex if:

\begin{itemize}

\item \textbf{i) } if $S \in \Sigma \ \ \Rightarrow
 \ \ \#S\leq 3$ and if $ A \subseteq S \Rightarrow A \in \Sigma$. We call $S$ a simplex.

\item \textbf{ii) } $\bigcup_{S \in \Sigma} S = V$.

\end{itemize}

\end{definition}

We Can imagine these simplices $S \in \Sigma$ to be an instruction describing which vertices to consider as connected. These are then used to define what points, edges and faces (triangles) are part of our simplicial complex. So if ...

\begin{itemize}
\item ... $ \#S = 1 \Rightarrow $ It is just a point representation of the vertex. By property 2, every vertex $v \in V$ appears in some element of $\Sigma$.

\item ... $ \#S = 2 \Rightarrow $ there exists an edge between the two vertices of S. For each $\{i,j\}\in \Sigma$ there are edges $e_{ij}$ and $e_{ji}$ that point in opposite directions.

\item ... $ \#S = 3 \Rightarrow $ there exists a triangle between the three vertices of S. By propertry \textit{i)} the edges and points around a face are always part of the simplicial complex.
\end{itemize}

We could have defined simplicial complexes for arbitrary dimensions $k \in \mathbb{N}$ but it is now best to focus only on what we need. See figure (\ref{fig:simplicial complex}) for an example. A simplicial complex is called \emph{closed} if every edge is part of exactly two faces.

\begin{figure}
\centering

\begin{tikzpicture}

  \coordinate [label={below right:$1$}] (1) at (0, 0);
  \coordinate [label={above right:$2$}] (2) at (2cm, 3cm);
  \coordinate [label={below left:$3$}] (3) at (-2.5cm, 1.6cm);
  \coordinate [label={above left:$4$}] (4) at (-2.3cm, +2.8cm);
  \coordinate [label={below right:$5$}] (5) at (2.3cm, +0.8cm);

  \draw [fill=lightgray!50,opacity=0.25]  (1.center) --  (2.center) --  (4.center) -- (3.center) -- cycle;
  
  \draw [very thick] (1) -- (5) -- (2) -- (4) -- (3) -- (1) -- (2) -- (3);

\end{tikzpicture}

\caption{Example of simplicial complex $(V,\Sigma)$ with 4 vertices, 7 edges and 2 faces.}
\label{fig:simplicial complex}
\end{figure}
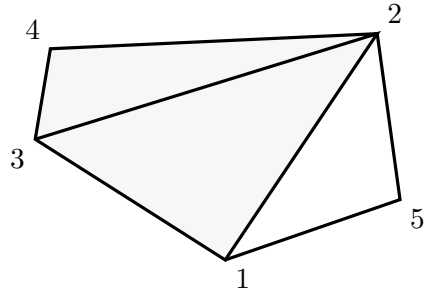

\medskip

\quad Next up are the sections on surfaces, which require careful attention in their definition. We might as well have focused on tangent vector fields. However, dealing with sections is worth it to archive a more general results.

We need to define line bundles on our simplicial complex $M:= (V,\Sigma)$ since in this context we want the simplicial complex to be an approximation of some smooth surface and thus require additional information such as a connection. 

\begin{definition}[Discrete hermitian line bundle]\label{discrete hermitian line bundle}

A discrete hermitian line bundle $L$ over a simplicial complex $M$  is a map $\pi : L \longrightarrow V$ such that $L_i:= \pi^{-1}({i})$  has the structure of a  1-dimensional hermitian vector space for each vertex $i \in V$.

\end{definition}

When saying \emph{hermitian line bundle $L$ over $M$} we implicitly have some $\pi$. At the moment there is no way to compare vectors between $L_i$ and $L_j$ as we have no notion of a transport in between neighbouring vertices. So the line bundle itself in the discrete case requires additional information or help to be more successfully understood as an approximation of a smooth surface. The following definition provides this extra help in form of mappings.

\begin{definition}[Discrete hermitian line bundle with curvature]\label{discrete hermitian line bundle with curvature}

A discrete hermitian line bundle with curvature over a simplicial complex $M$ is a triple $(L,\eta,\Omega)$ where $L$ is a hermitian line bundle over $M$. $\eta$ is a discrete connection mapping for each edge $e_{ij}$ to a unitary map $\eta_{ij}:L_i\longrightarrow L_j$ with $\eta_{ji} = \eta_{ij}^{-1}$. $\Omega$ is a closed real-valued 2-form such that on each face $ijk$ of the simplicial complex we have:

$$
\eta_{ki}\circ \eta_{jk}\circ \eta_{ij} = e^{i\Omega_{ijk} } \mbox{id}
$$

\end{definition}

The above definition is inspired by it's smooth equivalent where a connection $\nabla$ is used to define unique unitary parallel transports $\eta_{\gamma}$ on a path $\gamma$ \cite{PK:15:CLB} which is why we name $\eta$ a \emph{discrete connection}. The problem with the discrete connection is that the effect made by its transport $\eta_{ki}\circ \eta_{jk}\circ \eta_{ij}$ around the face $ijk$ is only observable up to mod $2\pi$. In other words, when transporting $x\in L_i$ using $\eta_{ki}\circ \eta_{jk}\circ \eta_{ij}$, we can not see how often the smooth equivalent would have been winding around the face $ijk$.

The solution to this problem is the additional information given by $\Omega$. $\Omega_{ijk}$ is the final rotation angle around a triangle and can be interpreted as a form of curvature. Without it the transport $\eta_{ki}\circ \eta_{jk}\circ \eta_{ij}$ would trap the rotation to [$-\pi,\pi$]\label{trapped curvature}. In our later calculations we will take $\Omega$ as an extra input to have a more precise solution (chapter: \ref{chap:Discrete Index Approach} \& \ref{chapter:Integral Geometry Approach}). 

If we are only given a discrete geometry (only a simplicial complex $M$) without any connection $\eta$ or curvature $\Omega$ we can still define a trivial one based on weak assumptions. In this case we will first define a new $\eta$ just as constructed in \cite{Knoppel:2013:GOD} (section 6). Then we define $\Omega_{ijk}\in (-\pi,\pi)$ to be as such that $\eta_{ki}\circ \eta_{jk}\circ \eta_{ij} = e^{i\Omega_{ijk} } id$. This causes the problem that $\Omega$ is trapped in an interval but this should not be surprising given the nature of starting with a finite set of points to approximate infinitely many where naturally information does get lost. The finer the triangulation the less relevant to curvature limitations will become since smaller triangles have less curvature anyway.

We avoid the case where $\Omega=\pm\pi$ is as it is confusing and meaningless as it happens extremly rarely and will pose no problem in this thesis.

The line bundle that we will handle is inspired by the tangent bundle $\mbox{T}M$. So $L_i = \mbox{T}_{i}M$ naturally has the structure of a two dimensional euclidean plane, but we will identify it using $\mathbb{C}$ and from now on work with $\mathbb{C}$. In terms of projective geometry we regard $\mathbb{C}$ as a \emph{line}, which is why we name $L$ an hermitian line bundle. 

\section{Sections and Projections}
Let us now be precise of what our \emph{vector fields} generalization for line bundles are.
\begin{definition}[Section]
A section $\phi$ over the hermitian line bundle $L$ is a map

$$
\phi : V \longrightarrow L \ ,\ i \mapsto \phi_i \in L_i
$$

We define $\Gamma(L)$ as the set of all sections over $L$.

\end{definition}

\label{base section} We can now choose a normalized and non-vanishing section $X$ and write any other section $\varphi \in \Gamma(L)$ as $\varphi_i = z_iX_i, \ z_i \in \mathbb{C}$ and call this $X$ our \emph{base section}. We can define any sections in $\Gamma(L)$ by selecting $z = (z_1, ... , z_n) \in \mathbb{C}^n$ \label{duality} and defining a new section $\varphi_i = z_iX_i$. Hence we can regard sections of finite simplicial complexes as elements of $\mathbb{C}^n$ and thus work with matrices when dealing with linear operators on sections. So with a choice of a normalized and non-vanishing $X \in \Gamma(L)$ we have an bijective identification between hermitian line bundle sections $\Gamma(L)$ and $\mathbb{C}^n$. This choice of base section is important for future calculations.

To stay true to the smooth analog of parallel transports we define a section $\phi \in \Gamma(L)$ to be \emph{parallel} if for the underlying line bundle we have $\eta_{ij}(\phi_i) = \phi_j$ \label{parallel}. The name parallel will become evident when we later define the rotation form (Def: \ref{rotation form}). A \textit{parallel transport} \cite{PK:15:CLB}  along a path of edges $\gamma := ij, \ jk, \ ... , \ hl$ is induced by the underlying connection $\eta$ and is defined as 

$$
P_{\gamma} : L_i \longrightarrow L_l , \ P_{\gamma}=\eta_{hl} \circ ... \circ \eta_{jk} \circ \eta_{ij}
$$

We also need to define the scalar product that we are going to use for the vector space of sections $\Gamma(L)$.

\begin{definition}\label{hermitian inner product}
Let $(L,\eta , \Omega)$ be a hermitian line bundle. For $\phi, \varphi \in \Gamma(L)$ we define a hermitian inner product based on the geometry of the simplicial complex.

$$
\langle \! \langle \ . \ , \ . \ \rangle \! \rangle : \Gamma(L) \times \Gamma(L) \longrightarrow \mathbb{C}
$$

$$
\langle \! \langle \phi  , \varphi \rangle \! \rangle = \sum_{i=1}^n A_i\langle\phi_i,\varphi_i\rangle_i
$$

where $A_i$ is the area of the polygonal surface around the vertex reaching the midpoints of the surrounding triangles and $\langle \ .\ , \ .\ \rangle_i$ is the scalar product on $L_i$.

\end{definition}

Lets us be clear about what we mean by \textit{zeros of vector fields with multiplicity} of sections $\varphi \in \Gamma(L)$. A normal zero is just a point $i \in V$ such that $\varphi_i = 0$. In the smooth case, the multiplicity of the zero is related to the behaviour of the vector field in a local neighbourhood of $v$ on the surface $M$. In the discrete case we can't look at a local neighbourhood of a point and we also cannot explicitly talk about zeros away from the vertices since we are only given information on the vertices. Hence we will approximate or make reasonable guesses using a fine triangulation and theorems relating the values on the vertices and the parallel transport to the zeros inside each face (chap: \ref{chap:Discrete Index Approach}),(chap: \ref{chapter:Integral Geometry Approach}).

Too more evidently comprehend what a zero of a vector field looks like we take look at figure (\ref{fig:Zebra_pattern}). We know that any stripe pattern of a closed surface has at least one zero (harry-ball theorem) and that we have a less trivial behaviour for the stripes at these zeros. The index of a zeros is the winding number of the directions in the local neighbourhood of the zeros.

\begin{figure}
\centering
\includegraphics[width=9cm]{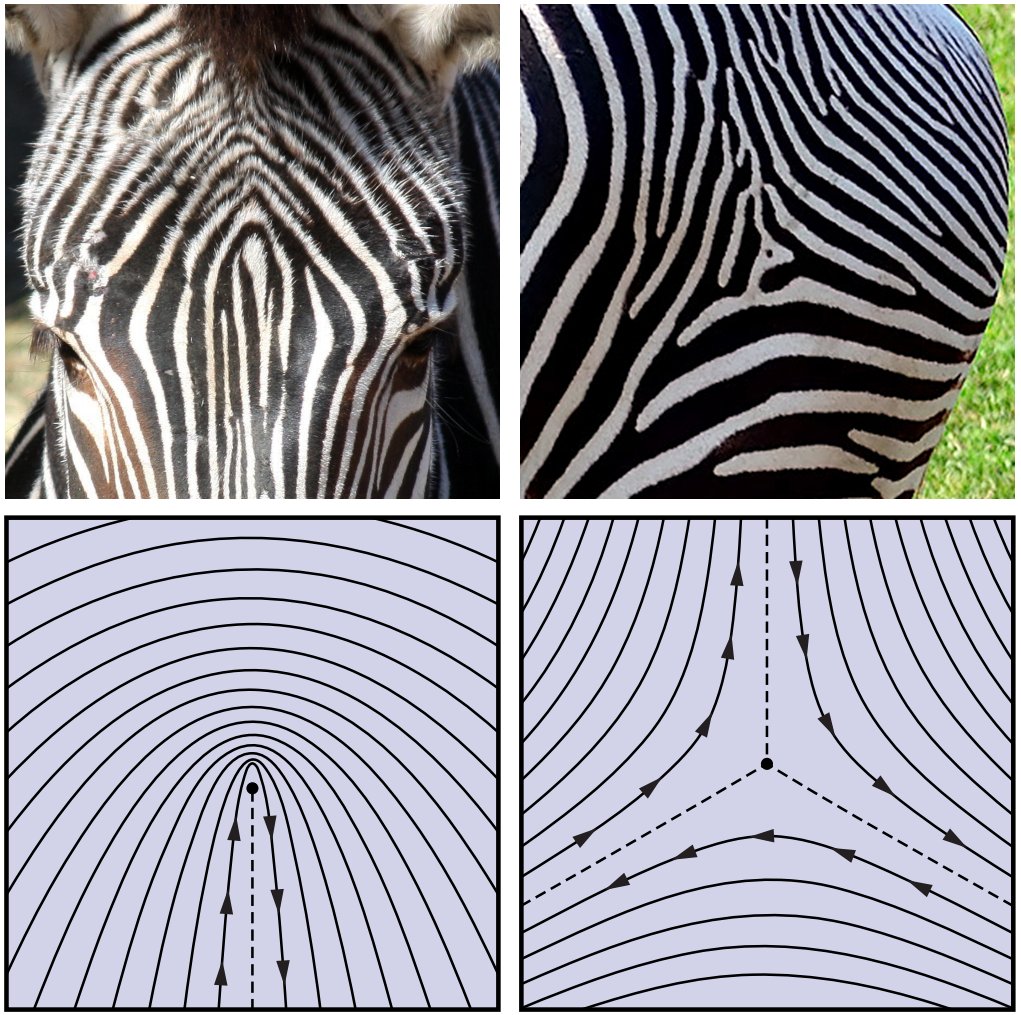}
\caption{Real life comparison of a Zebra's stripe pattern and a zero in a directional field \cite{Knoppel:2015:SPS}. (Photos courtesy Trisha Shears and André Karwath)}
\label{fig:Zebra_pattern}
\end{figure}

\quad The goal is to observe the distribution of the zeros of sections. So given a section $\phi \in \Gamma(L) \cong \mathbb{C}^n$ we can multiply that section by any $\lambda \in \mathbb{C}^*:=\mathbb{C}$ \textbackslash $\{0\}$ (point wise) so that the resulting section, defined as $\lambda \phi$, still has the exact same zeros with the same multiplicity and sign. This means that sections on hermitian line bundles that are non-zero multiples of each other can be regarded as part of the same equivalence class regarding the location, multiplicity and sign of their zeros, which is why projective geometry becomes relevant\label{projective argument}. Visually, multiplying a section $\phi$ by $\lambda \neq 0$ scales and rotates all ``vectors`` (the values $\phi_i$) by the same amount inside their respective spaces ($L_i$).

\begin{definition}[Projection map]\label{projection map}
If $z \in \mathbb{C}^n\setminus\{0\}$ for some $n \in \mathbb{N}$, then we define $[z]:=\{ \lambda X \in \mathbb{C}^n | \lambda \in \mathbb{C} \}$. This is the unique one dimensional subspace containing in $\mathbb{C}^n$ containing $z$. For future use we define the map

$$
[ \ \ ] : \Gamma(L)\setminus\{0\} \longrightarrow \mathbb{CP}^{n-1}
$$

$$ \phi \mapsto [\phi]$$

\end{definition}

$\mathbb{CP}^{n-1}$ is a orientable Riemannian manifold with the Fubini-Study metric\label{Fubini-Study metric} induced by the euclidean metric when relating $\mathbb{CP}^{n-1}$ to $S^{2n-1}$. Let $[x], [y] \in \mathbb{CP}^{n-1}$, then the $\mathbb{CP}^{n-1}$ distance between $[x],[y]$ is given by

\begin{equation}\label{eq:Fubini-Study metric}
d_F(x,y)= \mbox{arcos} \left ( \dfrac{\abs*{\langle \phi,\varphi \rangle}}{\abs*{\phi}\abs*{\varphi}} \right )
\end{equation}

with $\langle \ . \ , \ . \ \rangle$ being the standard hermitian scalar product of $\mathbb{C}^n$.Note how the maximum distance can only be $\frac{\pi}{2}$ and the area of $\mathbb{CP}^{n-1}$ is finite too \cite{Howard93thekinematic}.

\chapter{Laplace Operator \& Smoothing}\label{chapter:Laplace Operator}

\quad In this chapter we will closely examine the Laplace operator $\Delta$ with facts that are essential for the main results presented in this thesis. We will formally discuss $\Delta$'s properties and discuss its use in the smoothing operator (def: \ref{smoothing operator}).

\section{Laplace Operator}\label{sec:Laplace Operator}

\quad The Laplace operator $\Delta$ is an important linear function that appears in many equations in physics such as in heat and fluid flow, in quantum mechanics, in electrostatics, in gravitational potentials and in much more. In differential geometry it's most general form is called the Laplace-Beltrami operator defined by $f \mapsto \mbox{trace(Hessian(}f))$.

However, we require a different definition for our discrete constructions. Due to the fact that in the discrete case we have limited information of the smooth surface that our discrete surface tries to represent the discrete Laplace operator can also only be an approximation of the smooth Laplace operator. There exist multiple different possible approximations and none of them fulfil all core properties of the smooth Laplacian \cite{DLO:2007}. 

Nevertheless, our favorite Laplace operator is the cotan-Laplace operator as it was defined in \cite{pinkall1993computing}. For a simplicial complex $M=(V,\Sigma)$ this matrix $\Delta\in\mathbb{C}^{n\times n}$ is defined as

\[
\Delta_{ij}= 
\begin{cases}
0 \ , i\neq j, \ \{i,j\} \notin \Sigma\\[3mm]
\tfrac{1}{2}(\mbox{cot}(\alpha_{ij}) + \mbox{cot}(\beta_{ij}))\ , i\neq j, \ \{i,j\} \in \Sigma\\[3mm]
-\sum_{k\neq 1}\Delta_{ik} \ , \ i=j
\end{cases}
\]

using the angles in $\alpha_{ij},\beta_{ij}$ as shown in the figure (\ref{fig:cotan_laplace}).

\begin{figure}
\centering

\newcommand{\pythagwidth}{3cm}
\newcommand{\pythagheight}{2cm}
\begin{tikzpicture}

  \coordinate [label={below right:$l$}] (A) at (0, 0);
  \coordinate [label={above right:$j$}] (B) at (0, \pythagheight);
  \coordinate [label={below left:$i$}] (C) at (-\pythagwidth, 0);
  \coordinate [label={above left:$k$}] (D) at (-2.3cm, +2.8cm);

  \coordinate (D1) at (-\pythagheight, \pythagheight + \pythagwidth);
  \coordinate (D2) at (-\pythagheight - \pythagwidth, \pythagwidth);

\draw [fill=lightgray!25,opacity=0.25]  (A.center) --  (C.center) --  (D.center) -- (B.center) -- cycle;

  \draw [very thick] (C) -- (B) -- (A) -- (C) -- (D) -- (B);

    \draw (0,0.5cm) arc (90:180:0.5);
    \draw (-1.8cm, +2.6cm) arc (340:255:0.5) ;
    \node[] at (-1.8cm, +2.0cm)  {$\alpha_{ij}$};
    \node[] at (-0.6cm,0.6cm)  {$\beta_{ij}$};

\end{tikzpicture}

\caption{A part of a simplicial complex with labels needed to calculate the cotan-Laplace operator entry $\Delta_{ij}$.}
\label{fig:cotan_laplace}
\end{figure}
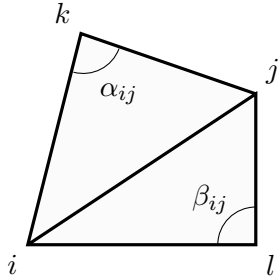

\begin{itemize}

\item $\Delta$ is linear over the field $\mathbb{C}$. This means that we can view $\Delta$ as a Matrix according to a non-vanishing base section $X\in\Gamma(L)$.

\item $\Delta$ is hermitian. $\Delta^*=\Delta$. ( $*=$ conjugate transpose ) Hence by the spectral theorem \cite{Hawkins19751} we know that all eigenvalues of $\Delta$ must be real and that we can always find a basis of orthogonal sections given the scalar product we defined.


\item $\Delta$ is positive semi definite, which is an equivalent statement as to say that all eigenvalues are greater or equal than zero. Together with the above statement this allows us to sort the eigenvalues and eigenvectors respectively. From here on we will index these as:

$\lambda_i \in \mathbb{R}_+\cup\{0\}$ is the eigenvalue to the eigenvector $v_i \in \mathbb{C}^n$ for $\forall i \in \mathbb{N}, \  i \leq n $ such that \label{eigenvalue ordering}

$$ \lambda_n \geq ... \geq \lambda_1 \geq 0$$


\end{itemize}

\subsection{Section Laplace Operator}\label{subsec: section laplace}

\quad In this thesis we will focus on the Laplace operator as defined in the related \emph{Globally Optimal Direction Fields} paper \cite{Knoppel:2013:GOD}(sec. 6.1.1). We omit repeating the construction details as they are too complex and are explained in detail in the paper.

This section Laplace operator is the adaptation of the cotan-Laplace operator as we defined in the section above to be used on sections. If the line bundle $L$ is trivial (i.e. $L_i\equiv\mathbb{C}$) then both, the cotan-Laplace operator and the adaption to line bundles are the same.

From now on throughout this thesis we will only use the adapted Laplace operator for the sections in $\Gamma(L)$ from a discrete hermitian line bundle $L$. In section (\ref{duality}) we had mentioned how $\Gamma(L)$ is identifiable to $\mathbb{C}^n$. For that reason let us now speak of the Laplace operator as a function $\Delta : \mathbb{C}^n \longrightarrow \mathbb{C}^n$ rather than using $\Delta : \Gamma(L) \longrightarrow \Gamma(L)$ after having chosen a non-vanishing base section $X\in\Gamma(L)$.

Additionally the adapted Laplacian from the paper does not have zero as an eigenvalue. This means that our ordering from above becomes:

$$ \lambda_n \geq ... \geq \lambda_1 > 0$$

See the comment about \emph{constant sections} in section (\ref{sec:Smoothing Operator}) to understand why.

\section{Smoothing Operator}\label{sec:Smoothing Operator}

\quad Now that we have established what we need of the Laplace operator, we can take a look at how it can be used for smoothing. For this reason, lets take a look at the heat equation \cite{Hahn:2012}.

Let $\varphi_t : I \times C(M) \longrightarrow C(M)$ be a smooth function of a surface $M$ and let $I$ be an open connected interval with $0 \in I \subset \mathbb{R}$. If $\Delta_s$ is the smooth Laplace-operator we say that $\varphi$ solves the heat equation if and only if 

$$\Delta_s \varphi_t = -\partial_t\varphi, \forall t\in I$$ 

For our discrete purpose we modify this to $\varphi_t:I\times \mathbb{C}^n\Rightarrow \mathbb{C}^n$ and use the discrete Laplace operator. The $i$-th entry in $\mathbb{C}^n$ represents value on the $i$-th vertex.

There are many solutions to it the way we defined it, but it gets interesting as soon as we select the initial conditions $\varphi_0 \in \mathbb{C}^n$ to be fixed. Physically, when using real scalar values on $M$, the solutions for real values describe the flow of heat on the given $2$ dimensional surface over time given the initial heat distribution $\varphi_0$ (fig: \ref{fig:heat flow}). The Laplace operator distributes values between two neighbouring vertices in proportion to the weight of the edge and the gradient of the values along the edge. Using complex instead of real values does not change the fact that these values become more evenly distributed each time when applying the Laplace operator.

\emph{Distribution values in proportion to the weight of the edge} works fine with scalars, but when using sections this becomes non-trivial. With an arbitrary selection a non-vanishing base section $X\in\Gamma(L)$ we need the transports $\eta_{ij}$ to compare $z_iX_i \in L_i$ with $z_jX_j\in L_j$ if $i$ and $j$ are two neighboring vertices. With enough smoothing, $\eta_{ij}(z_iX_i)$ becomes closer to $z_jX_j$ even though $z_i$ and $z_j$ can be completely different.

The important fact is that using the Laplace operator that we chose in section (\ref{subsec: section laplace}) we end up doing the same smoothing as a \emph{normal} Laplace operator would do. Section values become more evenly distributed where their edges have more weight dependant on the geometry, thus performing what we interpret as \emph{smoothing} of sections. The eigenvalues do not change when choosing a base section $X$ because that is equivalent to finding a different representation matrix according to a different base. Another difference is that there is no \emph{constant section} that is analogous to the constant solution for the scalar values on $M$ since parallel transports back to a vertex itself are not identity mappings. This is essentially why no eigenvalue is 0.

\begin{figure}
\centering
\includegraphics[width=4cm]{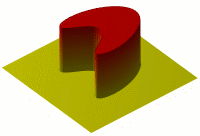}
\includegraphics[width=4cm]{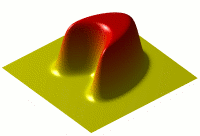}
\includegraphics[width=4cm]{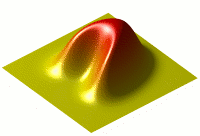}
\caption{Heat flow example on planar surface. Height represents the heat. Pictures from \cite{heatFlowGif:2007}}
\label{fig:heat flow}
\end{figure}

Note how $\varphi$ itself does not make any reference to the underlying geometry, yet every geometry has distinct solutions. For every different surface (simplicial complex) we acquire a different Laplace operator $\Delta$ which determines how the geometry affects the flow.

This crates a linear differential equation system. Similarly to the solution of the differential equation $$\lambda f(t)=\partial_tf(t), \ f(0)=A \ \ \  \text{  being  } \ \ \  f(t)=e^{-\lambda t}A \in \mathbb{C}$$ we know that $\varphi_t = e^{-t\Delta}\varphi_0$ solves the heat equation:

\begin{equation}
\begin{split}
\partial_t \left(e^{-t\Delta}\varphi_0 \right) & =\partial_t \left(\sum_{j=0}^{\infty}\frac{(-t)^j\Delta^j}{j!}\varphi_0\right)\\
 & = \sum_{j=0}^{\infty}\frac{-j(-t)^{j-1}\Delta^j\varphi_0}{j!}\\
 & = -\sum_{j=1}^{\infty}\frac{(-t)^{j-1}\Delta\Delta^{j-1}\varphi_0}{(j-1)!}\\
 & = -\Delta \sum_{j=0}^{\infty}\frac{(-t)^{j}\Delta^j}{j!}\varphi_0\\
 & = -\Delta \left( e^{-t\Delta}\varphi_0\right)\\
\text{and } e^{-0\Delta}\varphi_0 & = \varphi_0
\end{split}
\end{equation}

where $\Delta^j$ is defined as the $j$th successive application of the Laplace operator. Sadly, $e^{-t\Delta}$ is not easy compute explicitly, yet essential for this thesis. Let use therefor define:

\begin{definition}[Smoothing Operator]\label{smoothing operator}

Let $L$ be a hermitian line bundle on the simplicial complex $M$. $\forall t\in \mathbb{R}$ the smoothing operator $S_t : \Gamma(L) \longrightarrow \Gamma(L)$ is defined as:

$$ \varphi \mapsto S_t\varphi := e^{-t\Delta}\varphi = \sum_{j=0}^{\infty}\frac{(-t)^j\Delta^j}{j!}\varphi 
$$

\end{definition}

Intuitively, we can view $S_t$ as an operator that applies $\Delta$ a continuous amount of times to a section $\varphi$. For $t=0$ we apply $\Delta$ zero times and as $t$ increases so does the influence of higher powers of $\Delta$ in the operator as seen in equation (\ref{eq: smoothing powers}).

\begin{equation}\label{eq: smoothing powers}
e^{-t\Delta}\varphi_0 = \overbrace{\varphi_0}^\text{original field} + \overbrace{-t\Delta\varphi_0}^\text{1st order} + \overbrace{\dfrac{t^2}{2!}\Delta^2\varphi_0}^\text{2nd order} + \overbrace{\dfrac{-t^3}{3!}\Delta^3\varphi_0}^\text{3rd order} + \ ...
\end{equation}

Note that $S_t$ is as a composition of linear operators also linear. Now let us take a close look on how $S_t$ acts on the eigenvectos $\varphi_i$ of $\Delta$.

$$ S_t\varphi_i = \sum_{j=0}^{\infty}\frac{(-t)^j\Delta^j}{j!}\varphi_i = \sum_{j=0}^{\infty}\frac{(-t)^j\lambda_i^j}{j!}\varphi_i = \left(\sum_{j=0}^{\infty}\frac{(-t)^j\lambda_i^j}{j!}\right)\varphi_i = e^{-t\lambda_i}\varphi_i $$

 So $S_t$ has eigenvectors $\varphi_i$ to eigenvalues $e^{-t\lambda_i}$. Additionally $S_t$ is hermitian as seen here:

$$ S_t^* = \left(\sum_{j=0}^{\infty}\frac{(-t)^j\Delta^j}{j!}\right)^* = \sum_{j=0}^{\infty}\frac{(-t)^j\left(\Delta^j \right)^*}{j!} = \sum_{j=0}^{\infty}\frac{(-t)^j\left(\Delta^* \right) ^j}{j!} = \sum_{j=0}^{\infty}\frac{(-t)^j\Delta ^j}{j!} = S_t$$

By the spectral theorem again \cite{Hawkins19751} this means that $\{\varphi_i\}_{i=1,...,n}$ form an eigenvector basis to the eigen values $e^{-t\lambda_i}$ of the operator $S_t$. We write $\varphi \in \Gamma(L)$ as a composition of eigenvectors

$$ \varphi = \sum_{i=1}^{\infty}\mu_i\varphi_i, \ \mu_i \in \mathbb{C} $$

Thus we deduce:

\begin{equation}\label{eq: smoothing eigenvectorwise}
 S_t\varphi = S_t\left(\sum_{i=1}^{\infty}\mu_i\varphi_i\right) = \sum_{i=1}^{\infty}\mu_iS_t\varphi_i\\ = \sum_{i=1}^{\infty}\mu_ie^{-t\lambda_i}\varphi_i
\end{equation}

The expression (\ref{eq: smoothing eigenvectorwise}) allows us to work with $S_t$ without having to compute the powers of $\Delta$. Further above we had mentioned how for our purpose of finding the density of the indexed zeros, the non zero multiples of the sections becomes irrelevant (sec: \ref{projective argument}). So under the assumption that the above linear combination is well defined in terms of convergence we can always select a different representative section by normalizing.

\section{Smoothing Applications}

\quad Why do we want to smooth anything anyway? In the following section we will try to explain why it is useful besides computing the flow of heat. Lets explore the relationship between the $\Delta$ and the discrete Dirichlet energy.

In our discrete setting we can express the Dirichlet energy of a discrete section $\phi\in\Gamma(L)$ of a hermitian line bundle $L$ over a simplicial complex $M$ as

$$ E_D(\phi) := \langle \! \langle \Delta \phi , \phi \rangle \! \rangle $$

And by writing $\phi$ as a linear combination of eigenvectors of $\Delta$ we get

$$ E_D(\phi) = \sum_{i=1}^{\infty}|\mu_i|^2(\lambda_i) $$

To minimize this energy with the condition that $\langle \phi , \phi \rangle = 1$ it suffices to choose $|\mu_i| = 1$ for all $i$ such that $\lambda_i$ is the smallest eigenvalue and $\mu_i=0$ otherwise since $\lambda_i$ is always positive and ordered (sec: \ref{eigenvalue ordering}). So the minimizers of the Dirichlet energy under normalized condition all lie in the eigenspace of the smallest eigenvalue.

The most important source of inspiration for this work is the paper by Ulrich Pinkall, Felix Kn\"opel, Keena Crane and Peter Schr\"oder on globally optimized direction fields \cite{Knoppel:2013:GOD}. Figure (\ref{smooth bunny}) shows the output of the directional field algorithm when applied to a section on the bunny with highlighted zeros. The directional field on the bunny exemplifies one smoothed unit vector field.


\begin{figure}
\centering
\includegraphics[width=10cm]{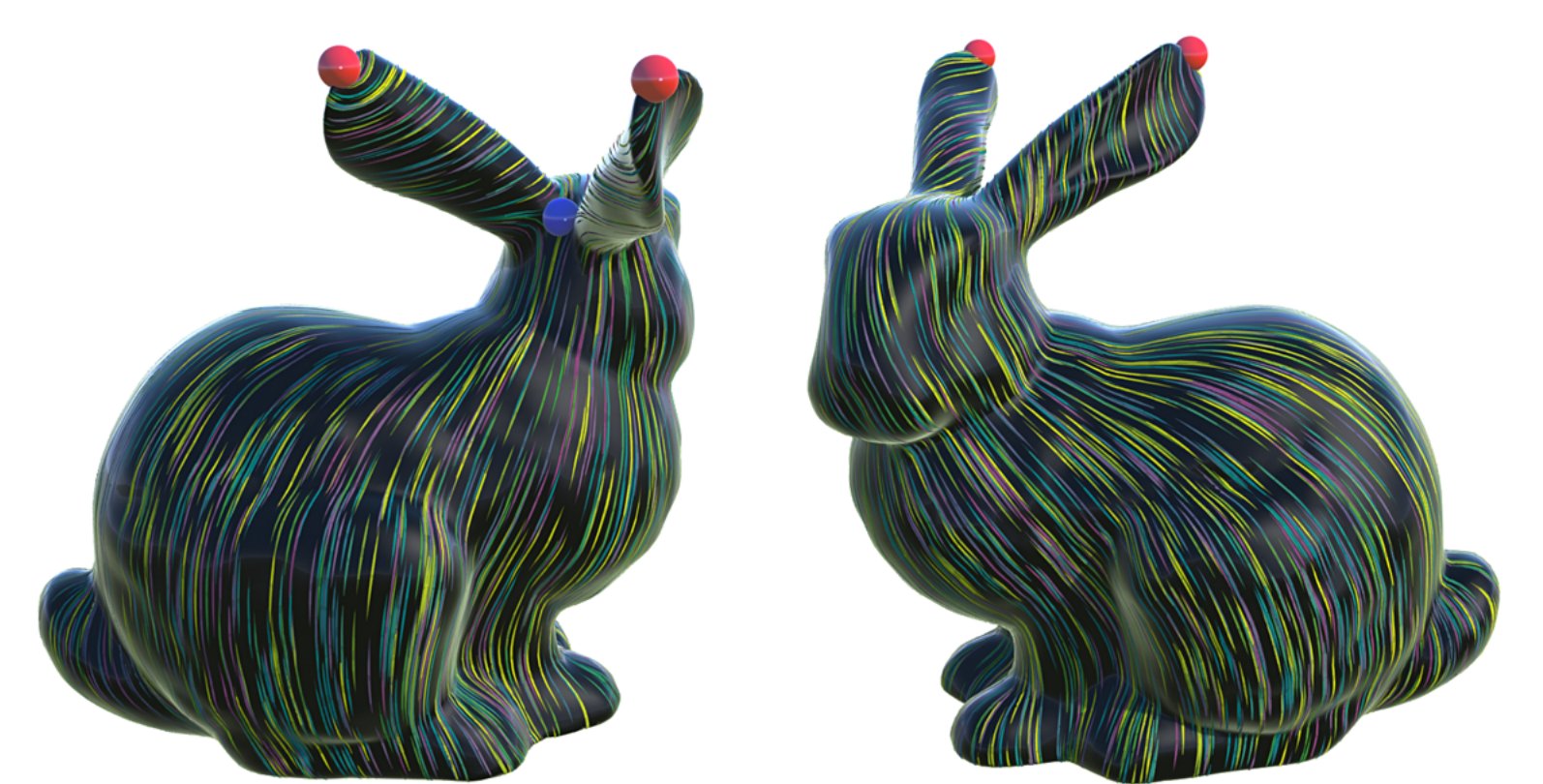}
\caption{globally optimized direction field example on a bunny. Colored dots indicate the index of a zero. \cite{Knoppel:2013:GOD}}
\label{smooth bunny}
\end{figure}

This image of the Bunny and several more examples allude that the zeros of these smoothed sections (i.e. the sources and sinks of the lines) gather in a non random fashion when smoothing random initial vector fields. However, instead of using the smoothing operator $S_t$, these examples where generated using a simpler computational approach by numerically solving a linear equations. This is called the \emph{inverse power method} and can be applied iteratively for advance smoothing. It aims at finding the eigenvector of the smallest (absolute) eigenvalue. These smoothed directional fields can be used to create nice stripe patterns and better texturzation (figure \ref{fig:smooth_texture}). For more information see \cite{Knoppel:2015:SPS}.

\begin{figure}
\centering
\includegraphics[width=9cm]{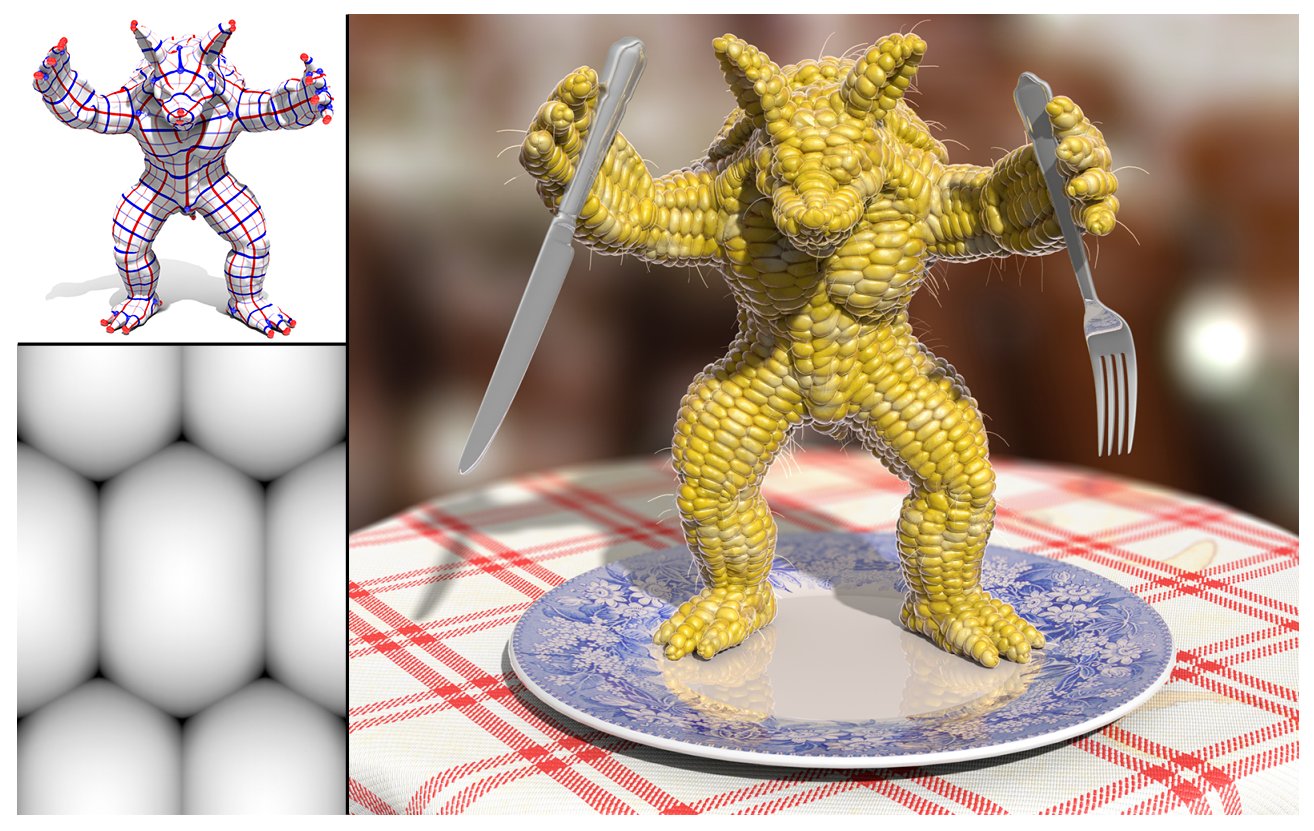}
\caption{Example application of smoothed sections \cite{Knoppel:2015:SPS}}
\label{fig:smooth_texture}
\end{figure}

 Another look at the Zebra figure (\ref{fig:Zebra_pattern}) sparks curiosity about their locations after smoothing. To generate a nice stripe pattern from a smoothed directional field we would prefer these zeros to be in less obvious locations. Let us find out where they are heading by the end of this thesis.

\section{Smoothing Convergence}

For large values of $t$ we can establish the following

\begin{theorem}[Smoothing convergence]\label{smoothing convergence}
 
$$
\text{If } \ \ \varphi = \sum_{i=1}^{\infty}\mu_i\varphi_i \neq 0
$$
for some $\mu_i \in \mathbb{C}$. Define $i_{\min} := \min\{ i \in \mathbb{N} \ | \ \mu_i \neq 0\}$. then we have

$$
\lim\limits_{t \to \infty} [S_t\varphi] = \left [ \left ( \sum_{i=i_{\min}}^{i_{\min}+\\dim(Eig(\lambda_{i_{\min}}))-1} \mu_i\varphi_i \right ) \right ]
$$

where $[ \ ]$ is the projection map [defined in sec. \ref{projection map}].

\end{theorem}

\begin{proof}[proof]

$S_t$ is hermitian and thus diagonalizable. Hence the algebraic multiplicity is equal to the geometric multiplicity for each eigenvalue. We know that $i_{\min}$ is well defined since otherwise $\varphi \equiv 0$ which is not the case. Therefore we can write

$$
S_t\varphi = \sum_{i=i_{\min}}^{\infty}\mu_ie^{-t\lambda_i}\varphi_i
$$

We can scale $\varphi$ by non zero complex numbers since we are only interested in the corresponding element in $\mathbb{CP}^{n-1}$. Therefore we define a scaled version of this linear combination 

$$
\psi_t := \dfrac{S_t\varphi}{e^{-t\lambda_{i_{\min}}}}= \sum_{i=i_{\min}}^{\infty}\mu_i e^{t(\lambda_{i_{\min}}-\lambda_i)}\varphi_i
$$

and compute the limits of the new linear factors

\[
\lim\limits_{t \to \infty} e^{t(\lambda_{i_{\min}}-\lambda_i)} = 
\begin{cases}
\infty \ , \ \lambda_{i_{\min}} - \lambda_i > 0\\[3mm]
1 \ , \ \lambda_{i_{\min}} - \lambda_i = 0\\[3mm]
0 \ , \ \lambda_{i_{\min}} - \lambda_i < 0
\end{cases}
\]

The ``$\infty$`` case can not occur since the eigenvalues are ordered (sec: \ref{eigenvalue ordering}) and $\lambda_{i_{\min}}$ is the smallest occurring eigenvalue. So if $t \rightarrow \infty$ only those vectors $\varphi_i$ become relevant that have $\lambda_i = \lambda_{i_{\min}}$. Due to the ordering of the eigenvectors those are exactly $\varphi_{i_{\min}}, \varphi_{i_{\min}+1}, ... , \varphi_{ i_{\min}+\dim(Eig(\lambda_{i_{\min}}))-1}$ which we use to conclude our proof:

\begin{equation}
\begin{split}
\lim\limits_{t \to \infty} [S_t\varphi] & = \lim\limits_{t \to \infty} [\psi_t]\\
& =  \left [ \lim\limits_{t \to \infty} \sum_{i=i_{\min}}^{\infty}\mu_i e^{t(\lambda_i-\lambda_{i_{\min}})}\varphi_i \right ]\\
& = \left [ \sum_{i=i_{\min}}^{i_{\min}+\dim(Eig(\lambda_{i_{\min}}))-1}\mu_i \varphi_i \right ]
\end{split}
\end{equation}

\end{proof}

This proposition lets us now deduce that for large $t$, the $\varphi_i$ components that have an eigenvalue equal to the smallest eigenvalue $\lambda_{i_{\min}}$ have much larger linear factors than any other eingevector, thus rendering every other eigenvector's influence obsolete. Since in this thesis $\varphi$ is randomly selected we have $\mu_i$ from continuous distribution, therefore it is safe to say that the probability that $\mu_1 = 0 $ is zero. This results in $i_{\min}=1$ and thus

$$
\lim\limits_{t \to \infty} [S_t\varphi] = \left [ \sum_{i=1}^{\dim(Eig(\lambda_{1}))} \mu_i\varphi_i \right ]
$$
 
The random section $\phi$ is mapped orthagonally onto $Eig(\lambda_1)$. Because of this the choice of orthonormal basis of eigenvectors is irrelevant in the final computation.

\chapter{Discrete Index Approach}\label{chap:Discrete Index Approach}

\quad In the previous chapter we established and discussed smoothing, now it is time to derive and deal with ways to compute the expected sum of signed zeros with multiplicity on discrete surfaces. In this chapter we will discuss one way of doing so using the rotation 1-form $\xi$ and face indexes.

On smooth surfaces, the Poincare-Hopf theorem relates the sum of indexes with a basic property of the surface. In this chapter we will expand this relationship between the geometry and its indexes for the discrete case in order to work with it on a per triangle basis. The discrete index theorem that we will derive here will explain how the number of zeros on a simplicial complex with a hermitian line bundle can be approximated without further information of the inside. We will also see the limits of this approach before continuing with a completely different one in the chapter afterwards. 


\section{Rotation Form and Indexes}

\begin{definition}[Discrete rotation form $\xi$]\label{rotation form}

Let $\phi \in \Gamma(L)$ a section on a simplicial complex $M$ and $ij$ and edge in $M$ and $\eta$ the connection of $L$. If $\eta_{ij}(\phi_i) \neq -\phi_j$ and $\phi_i\neq 0$ the discrete rotation 1-form $\xi$ is defined by

$$
\xi^{\phi}_{ij} = \arg \left (\dfrac{\phi_j}{\eta_{ij}(\phi_i)} \right ) \in (-\pi,\pi) 
$$

\end{definition}

This is discrete analog of the smooth rotation form $\xi^{\phi}:=\tfrac{\langle \nabla \phi,i\phi \rangle}{\langle \phi , \phi \rangle}$ where $\nabla$ is a connection. The rotation form measures the rotation caused by the unitary maps $\eta_{ij}$ along each edge, like when transporting a tangent vector from one tangent space $\mbox{T}_pM$ to another tangent space $\mbox{T}_qM$ when $p,q$ are elements of a smooth manifold $M$.

$\eta_{ij}(\phi_i) \neq -\phi_j$ and $\phi_i\neq 0$ is required to avoid problems with $\xi$ being undefined. For our later purpose this will not be a problem since the set of sections that fulfill to cause these problems form a null set meaning that their influence will vanish later when taking integrals. 

Also note how parallel sections defined in (sec. \ref{parallel}) influence the rotation form. If $\phi$ is a parallel section, then $\xi_{ij}^{\phi}=\arg\left(\frac{\phi_j}{\phi_j}\right)=\arg(1)=0$ for all edges $ij$ in $M$. Parallel sections do not rotate.

\begin{lemma}\label{rotation invariance}

The discrete rotation $\xi$ form is invariant under non-zero multiplications.

\end{lemma}

\begin{proof}
If $(L,\eta,\Omega)$ is hermitian line bundle on a simplicial complex $M$, $ \phi \in \Gamma(L)$, $c \in \mathbb{C}\setminus\{0\}$. Let $ij$ be an edge in $M$, then 
$$\xi_{ij}^{c\phi} = \arg \left (\dfrac{c\phi_j}{\eta_{ij}(c\phi_i)} \right ) = \arg \left (\dfrac{c\phi_j}{c\eta_{ij}(\phi_i)} \right ) = \arg \left (\dfrac{\phi_j}{\eta_{ij}(\phi_i)} \right ) = \xi_{ij}^{\phi}$$

\end{proof}


We use this rotation form to define the discrete index, which is another analog to the smooth index as it appears in  \cite{milnor1997topology}.

\begin{definition}[Discrete index]

Let $(L,\eta,\Omega)$ be a discrete hermitian line bundle over the simplical complex $M$, $\phi \in \Gamma(L)$ and $ijk$ a triangle in $M$. We define the index 2-form of $\phi$ by 

$$
\mbox{ind}^{\phi}_{ijk} := \dfrac{1}{2\pi}(d\xi^{\phi}_{ijk}+\Omega_{ijk})
$$

where $d\xi^{\phi}_{ijk} = \xi^{\phi}_{ij} + \xi^{\phi}_{jk} + \xi^{\phi}_{ki}$

\end{definition}

The index describes the change in angle a section undergoes after being transported around a triangle while taking into consideration the curvature $\Omega_{ijk}$ of the face. In the smooth case if we would compute the index in the neighbourhood of an isolated zero it would return us the index of the zero. However, being in a discrete setting, a triangle is the smallest possible neighbourhood we can compute the winding number from. The rotation form $\xi$ can also not determine how many times $\eta_{ij}$ rotated $\phi_i$ along the edge $ij$ as it is bound to $(-\pi,\pi)$.

\begin{lemma}\label{lemma: index invariance}

If $(L,\eta,\Omega)$ is hermitian line bundle on a closed simplicial complex $M$, $ \phi \in \Gamma(L)$, $c \in \mathbb{C}\setminus\{0\}$. Let $ijk$ be a triangle in $M$. Then 
$\mbox{ind}^{c\phi}_{ijk}=\mbox{ind}^{\phi}_{ijk}\in\mathbb{Z}$ 

\end{lemma}

\begin{proof}

Every unitary map $\eta : \mathbb{C} \longrightarrow \mathbb{C}$ rotates each individual element $\phi_i$ precisely by the amount given in the rotation form. This way it can be shown that $\eta_{ki}\circ\eta_{jk}\circ\eta_{ij}=e^{i\xi_{ij}^{\phi}+i\xi_{jk}^{\phi}+i\xi_{ki}^{\phi}}$

$$\exp(i2\pi \mbox{ind}_{ijk}^{\phi}) = \mbox{exp}(i(\xi_{ki}^{\phi} + \xi_{jk}^{\phi} + \xi_{ij}^{\phi}) + i\Omega_{ijk}) = \dfrac{e^{i(\xi_{ki}^{\phi} + \xi_{jk}^{\phi} + \xi_{ij}^{\phi})}}{e^{-i\Omega_{ijk}}}=...$$

$$...=\dfrac{e^{i\xi_{ij}^{\phi}}e^{i\xi_{jk}^{\phi}}e^{i\xi_{ki}^{\phi}}}{e^{-i\Omega_{ijk}}}=
\dfrac{\eta_{ki}\circ\eta_{jk}\circ\eta_{ij}}{e^{-i\Omega_{ijk}}}=
1 \Rightarrow i2\pi \mbox{ind}^{\phi}_{ijk} =
0 \ (\mbox{mod} \ 2\pi i)\Rightarrow \mbox{ind}^{\phi}_{ijk} \in \mathbb{Z}$$

where we used the definition of the curvature 2-form from $\Omega$ (def: \ref{discrete hermitian line bundle with curvature})  in the last line.

To show the invariance under multiplication by $c$ we just need to apply lemma \ref{rotation invariance} inside the definition.

$\mbox{ind}^{c\phi}_{ijk} = \dfrac{1}{2\pi}(d\xi^{c\phi}_{ijk}+\Omega_{ijk})= \dfrac{1}{2\pi}(d\xi^{\phi}_{ijk}+\Omega_{ijk}) =  \mbox{ind}^{\phi}_{ijk}$

\end{proof}

If we are given a simplicial simplex but no curvature and we we define our own one (sec: \ref{trapped curvature}) then $\xi_{ki}^{\phi},\xi_{jk}^{\phi},\xi_{ij}^{\phi} \in (-\pi,\pi], \ \ \Omega_{ijk}\in (-\pi,\pi)$ and thus $\mbox{ind}^{\phi}_{ijk}\in\{-1,0,1\}$ as the only possible solutions.

\section{Section Zeros and Density}

\quad The link to the zeros of vector fields is established by a discrete analogous version to the Poincare-Hopf index formula that we will now establish for closed surfaces.

\begin{theorem}[Discrete closed index formula]

Let $(L,\eta,\Omega)$ be a discrete hermitian line bundle on a closed simplicial complex $M$, $\phi \in \Gamma(L)$. Define $F:=\{ijk\ : \ ijk \mbox{ is face of } M\}$. Then  

$$\sum_{ijk \in F}\mbox{ind}_{ijk}^{\phi}=
\dfrac{1}{2\pi}\sum_{ijk \in F}\Omega_{ijk}=:\deg(L)$$

\end{theorem}

\begin{proof}

$$\sum_{ijk \in F}\dfrac{1}{2\pi}\Omega_{ijk} = 
\sum_{ijk \in F} \left (\mbox{ind}_{ijk}^{\phi}-\dfrac{1}{2\pi}d\xi_{ijk}^{\phi}\right )=
\sum_{ijk \in F}\mbox{ind}_{ijk}^{\phi} -\dfrac{1}{2\pi}\overbrace{\sum_{ijk \in F}d\xi_{ijk}^{\phi}}^\text{=0}$$

The cancellation in the sum happens due to $\xi_{ji}^\phi=-\xi_{ij}^\phi$ and every edge orientation appearing once in the sum because $M$ is a closed surface.

\end{proof}

Since $M$ together with our hermitian line bundle $L$ mimics a smooth surface with connection, we are motivated to interpret the $\mbox{ind}_{ijk}^{\phi}$ as the signed sum of indexes of the zeros inside the triangle $ijk$ \label{discrete zeros}. The following remark links the discrete setting with the smooth setting to extract to smooth's setting interpretation.

\medskip

\begin{remark}

The main reason to justify that the discrete index $\mbox{ind}_{ijk}^{\phi}$ represents the signed sum of indices inside the triangle $ijk$ is the smooth analog theorem involving differential geometry that was proven in \cite{PK:15:CLB}.

\begin{tcolorbox}[colback=gray!10, colframe=gray!50, title={\textbf{Theorem}}]

Let $L$ be a smooth hermitian line bundle with connection $\nabla$ on a manifold $M$. Let $\Omega$ be its curvature form, $\phi \in \Gamma(L)$ a smooth section, $\xi^\phi$ its rotation form and $Z$ be the discrete set of zeros. If $C$ is a finite smooth 2-chain with $\partial C \cap Z = \emptyset$, then

$$2\pi\sum_{p\in C\cap Z}\mbox{ind}_p^\phi = \int_{\partial C}\xi^\phi + \int_C\Omega$$

\end{tcolorbox}

If we replace $C$ by the triangle $ijk$ we see the smooth analog to our discrete attempt. Even though the information acquired in the discrete setting is incomplete, by seeing how the smooth rotation form $\xi$ and the smooth curvature $\Omega$ relate to the indices we can mimic this with our discrete replacements of $\xi$ and $\Omega$.

\end{remark}

\medskip

In the discrete setting if we are not given $\Omega$, then $\mbox{ind}_{ijk}^\phi$ is only a weak approximation since a mesh that has not been triangulated fine enough would group zeros together and possibly make them cancel each other out or raise the index above the measurable because of the bounded curvature (sec: \ref{trapped curvature}). That however, is the inevitable nature of discrete set ups that lost information in between the vertices and a replacement can only be guessed using reasonable assumptions. But if we are given the 2-form $\Omega$ this becomes a better approximation.
 
Hence we can finally define what we wanted to compute all along:

\begin{definition}[Density]\label{Density}

Let $L$ be a hermitian line bundle on a closed oriented simplicial complex $M$ and $\phi \in \Gamma(L)$, then we define the \textit{zero density of $\phi$ on the face $ijk$ of $M$} as

$$Z_{ijk}(\phi) := \dfrac{\mbox{ind}_{ijk}^\phi}{\mbox{Area}(ijk)}$$

\end{definition}

Once again, the finer the triangulation the mesh is the more accurate this density becomes to the smooth case. This expression is also invariant under non-zero section multiples $c\phi,\ \ c\in \mathbb{C}\setminus\{0\}$.

\section{Random Section Smoothing}\label{section:Random Section Smoothing}

\quad Now the time has to come to unite the tools established in this and the previous chapters to attempt to fulfill this thesis' purpose: to compute the expected sum of indexes on each triangle of smoothed random hermitian line bundle sections on a 2-dimensional surfaces. 

\begin{definition}[Random Section]\label{random section}

Let $\{\varphi_i\}_{i=1,...,n}$ form an orthonormal basis of eigenvectors of $\Delta$. And let $X$ be a non-vanishing normalized section $\in\Gamma(L)$. $\phi = \sum_{i=0}^{n}\phi_i\varphi_i \in \Gamma(L)$ is our random section if $\langle\phi,\varphi\rangle=x_l+iy_l \in \mathbb{C}$ where $x_l$ and $y_l$ are real, independent and normal distributed. $z_l:=x_l+iy_l$ is the random variable of the sum and has $p:\mathbb{R}\longrightarrow\mathbb{R}$ as its probability density function.

\end{definition}

Since $\{\varphi_i\}_{i=1,...,n}$ forms an orthonormal basis we can be certain that the resulting $\phi$ will be as randomly distributed as with any other orthonormals basis such as the canonical one. The choice of the eigenvector basis will make our computations easier. Therefore smoothing a random section simplifies to:

$$S_t\phi =
\sum_{i=1}^n\phi_i S_t\varphi_i=
\sum_{i=1}^n e^{-t\lambda_i}z_i\varphi_i$$

So for fixed $t>0$ the density of indexed zeros in the triangle $ijk$ of a random section can be computed by

$$Z_{ijk}(S_t\phi)=Z_{ijk} \left ( \sum_{l=1}^n e^{-t\lambda_l}z_l\varphi_l \right )$$

which is why, by integrating over $\mathbb{C}$ for each $z_l$, the expected density of indices can be expressed as:

$$\int \left(\prod_{l=1}^{n} p(z_l) \right) Z_{ijk} \left ( \sum_{l=1}^n e^{-t\lambda_l}\phi_l\varphi_l\right ) dz_1...dz_{n}$$

Returning to our desire to let $t$ run to $\infty$ we can apply the convergence result established in theorem (\ref{smoothing convergence}). This is possible due to the fact that non-zero complex multiples of sections do not change the number of zeros as shown in lemma \ref{lemma: index invariance} and thus any representative of $[S_t\phi]$ has the same zeros. Additionally, since the $(Z_l)$ arise by normal distributions, the probability that $(Z_1)=0$ is zero. Therefore $i_{\min}:= \min\{ i \in \mathbb{N} \ | Z_i \neq 0\}$ equals to 1 with probability 1 and in the density computation from the theorem we can calculate with $i_{\min}=1$. Hence we simplify the density to:

$$\int \left(\prod_{l=1}^{n} p(z_l) \right ) Z_{ijk} \left ( \sum_{l=1}^{\dim(Eig(\lambda_{i_{\min}}))} e^{-t\lambda_l}z_l\varphi_l\right ) dz_1...dz_{n}$$

This expression only handles one eigenvalue multiple times ($\lambda_{1}=...=\lambda_{\dim(Eig(\lambda_{i_{\min}}))}$). So we can move $e^{-t\lambda_l}$ in front of the sum and then remove it from the density function since $e^{-t\lambda_l}\neq 0$ and non-zero multiples do not affect the number of zeros of a (lemma \ref{lemma: index invariance}). Those integrals who's $z_l$ influence vanishes integrate to 1.

\begin{equation}\label{eq:probability density}\int \left(\prod_{l=1}^{\dim(Eig(\lambda_{i_{\min}}))} p(z_l) \right ) Z_{ijk} \left ( \sum_{l=1}^{\dim(Eig(\lambda_{i_{\min}}))} z_l\varphi_l\right ) dz_1...dz_{\dim(Eig(\lambda_{i_{\min}}))}
\end{equation}




In most cases, especially when the geometry was not generated by some parametric equations, all eigenvalues are unique (as in a random sample from a continous spectrum). Therefore we will usually face $\dim(Eig(\lambda_1))=1$.

Assume that $\dim(Eig(\lambda_1))=1$ which would again simplifies the computations to:

$$\int p(z_1) Z_{ijk} \left ( z_1\varphi_1\right ) dz_1 = \int p(z_1) Z_{ijk} \left ( \varphi_1\right ) dz_1 = Z_{ijk} \left ( \varphi_1\right )$$

This can be interpreted as a deterministic behaviour of the zeros of smoothed sections on closed surfaces whose smallest eigenvalue $\lambda_1$ only appears once. The resulting expected density of indexed zeros for each triangle $ijk$ can then be combined to view globally where the zeros distribute.

If $\dim(Eig(\lambda_1))>1$ we do not get a deterministic behaviour, but rather an even distribution along the possible zeros with linear combinations from $Eig(\lambda_1)$. However, I did not manage to simplify expression (\ref{eq:probability density}).

At last let us think about what happens if we smooth nothing i.e. set $t=0$ in $S_t$. Since the area is independent of $t$ let us focus on the average index on a triangle $ijk$. Let $p(\phi)$ be the probability distribution of sections dependant on $Z_1,...,Z_n$ as in the definition above.

\begin{align*}
\int \left(\prod_{l=1}^{n} p(z_l) \right) \mbox{ind}_{ijk} \left ( \sum_{l=1}^n \phi_l\varphi_l\right ) dz_1...dz_{n} 
&= \int_{\Gamma(L)} p(\phi)\mbox{ind}_{ijk}^\phi d\phi \\
&= \dfrac{1}{2\pi} \int_{\Gamma(L)} p(\phi) \left( \xi^{\phi}_{ij} + \xi^{\phi}_{jk} + \xi^{\phi}_{ki}+\Omega_{ijk}\right ) d\phi \\
&=        \dfrac{1}{2\pi}   \underbrace{\int_{\Gamma(L)}p(\phi)\xi^{\phi}_{ij}d\phi}_{=0}\\
& \ \ \ + \dfrac{1}{2\pi} \underbrace{\int_{\Gamma(L)}p(\phi)\xi^{\phi}_{jk}d\phi}_{=0}\\
& \ \ \ + \dfrac{1}{2\pi} \underbrace{\int_{\Gamma(L)}p(\phi)\xi^{\phi}_{ki}d\phi}_{=0}\\
& \ \ \ + \dfrac{1}{2\pi}\int_{\Gamma(L)}p(\phi)\underbrace{\Omega_{ijk}}_{\text{const.}} d\phi \\
&=\dfrac{\Omega_{ijk}}{2\pi}\int_{\Gamma(L)}p(\phi)d\phi=\dfrac{\Omega_{ijk}}{2\pi}
\end{align*}

where the zeros happen due to $\xi_{ij}^\phi$ taking every possible value just as often positve as negative.

This shows that the average signed number of zeros of random not yet smoothed vectorfields are dependant on the curvature 2-form $\Omega$ as well as the area of the triangle $ijk$.

\begin{equation}\label{eq:unsmoothed_result}
\text{Average value of } Z_{ijk}(S_0\phi) \ \ \ = \ \ \ \dfrac{\Omega_{ijk}}{2\pi \mbox{Area}(ijk)}
\end{equation}

\chapter{Integral Geometry Approach}\label{chapter:Integral Geometry Approach}

\quad The discrete index approach from the previous chapter struggles with obtaining an evident solution for the cases where $\dim(Eig(\lambda_1))>1$. We need to find a different approach to the expression (\ref{eq:probability density}) that I did not manage to simplify. In this chapter we will cover an integral geometric approach to obtain an expression of the density of zeros with sign and multiplicity which serves well for $\dim(Eig(\lambda_1))>1$.

The derivation of the solution is in its idea similar to the proof of the Cauchy-Crofton formula \cite{santalo2004integral}, which links the number intersections of all lines with a curve to the length of the curve. We will now set up the tools to perform something similar. To do this properly with transparent intuition we will establish the basic idea first.

\section{Basic Idea}\label{Initial Approach}

\quad Let's take a non-vanishing normalized base section $X\in\Gamma(L)$. We also require special sections:

\begin{definition}[Delta section]\label{delta section}

The \emph{delta section} for a vertex $i$ from a simplicial complex $M$ is define as

$$ \delta_i \in \Gamma(L) \ : \ \forall \phi \in \Gamma(L) \Rightarrow \langle \! \langle \delta_i, \phi \rangle \! \rangle = \langle X_i,\phi_i\rangle \in L_i $$

\end{definition}

For $\Gamma(L)$ the set $\{\delta_i \ : \ i \in V \}$ is similar to the canonical basis $\{e_i \ : \ i=1,...,n \}$ of $\mathbb{C}^n$. According to the inner product from definition (\ref{hermitian inner product}) $\{\delta_i \ : \ i \in V \}$ is a orthonormal basis since 

\[
\langle \! \langle \delta_i ,\delta_j \rangle \! \rangle = 
\begin{cases}
1 \ , \ i=j\\[3mm]
0 \ , \ i\neq j
\end{cases}
\]

The linear function $\langle \delta_i , \ . \ \rangle$ also serves as en \emph{evaluator} at vertex $i$ since if $\phi=\sum_{i=1}^nz_iX_i\in \Gamma(L)$ then

$$\langle \! \langle \delta_i, \phi \rangle \! \rangle = \langle X_i,z_iX_i\rangle=z_i$$

Picking up a random section $\phi$ as in (def: \ref{random section}), using the scalar product we see that $\phi$ has a zero at $p$ if and only if $\langle \! \langle \delta_i , \phi \rangle \! \rangle = 0$. Now by smoothing $\phi$ to $S_t\phi$ and since $S_t$ is self adjoined we can express the scalar product as

$$0 = \langle \! \langle \delta_i , S_t\phi \rangle \! \rangle = \langle \! \langle \overbrace{S_t\delta_i}^{\psi_i^t:=} , \phi \rangle \! \rangle = \langle \! \langle \psi_i^t , \phi \rangle \! \rangle $$

Notice how $0=\langle \psi_i^t , \phi \rangle$ can be interpreted as $\phi$ being orthogonal to $\psi_i^t$. Essentially the condition that $S_t\phi$ is zero at $p$ is equivalent to the statement that $\psi_i^t$ is inside the hyperplane $\phi^{\perp}$ induced by $\phi$. As $\phi$ was completely random, we basically seek an expression of the average intersection volume between all hyperplanes and the section $\psi_i^t$. This expression will link this problem to integral geometry.

However, the vertex set is discrete and the chance of having a zero right on a vertex is zero. To keep the idea of identify zeros of faces with intersection with hyperplanes we are now required to do a lot more work to formally work with this.

Lets keep in mind (again) that section multiplications by non-zero factors don't affect the zeros. Hence we can focus on the intersections of $[\psi^t(p)]$ with all hyperplanes inside $\mathbb{CP}^{n-1}$. This is especially useful because inside $\mathbb{CP}^{n-1}$ all hyperplanes are oriented and unique, thus making the space of hyperplanes inside $\mathbb{CP}^{n-1}$ identifiable with itself.

From here on we will focus on the computation of the zeros for an arbitrary triangle in the simplicial complex $M$ made by the neighbouring vertices $ijk$. Many of the following results where archived with the help of Felix Knöppel’s notes \cite{IntegralGeometry:2015} and \cite{Kodaira:2015}.

\section{Embedding into $\mathbb{CP}^{n-1}$}\label{Embedding into}

\quad First we will find a suitable embedding of the surface $M$ of a simplicial complex with $n\in\mathbb{N}$ vertices into $\mathbb{CP}^{n-1}$ because of its later use when identifying the zeros. To do this, we will introduce the Kodaira correspondence taken from Franz Pedit's lecture\footnote{\url{https://www2.le.ac.uk/departments/mathematics/extranet/staff-material/staff-profiles/kl96/stuff/cseminar.pdf}}.

Let $L$ be a smooth hermitian line bundle. First we define for each $p\in M$ the \emph{evaluation map} $\mbox{ev}_p:\Gamma(L)\rightarrow L_p$ , $\psi \mapsto\psi_p$ that simply evaluates a section at the point $p$. The adjoint map is $\mbox{ev}_p^*:L_p^*\rightarrow\Gamma(L)$ and the Kodaira correspondence is then defined as the embedding

$$f:M\rightarrow \mbox{P}(\Gamma(L)^*), \ p\mapsto \mbox{\mbox{ker}}\, \mbox{ev}_p^*$$

Recall that if $\langle . , . \rangle :\Gamma(L)\times \Gamma(L) \rightarrow \mathbb{R}$ is the canonical pairing (if $\lambda\in L_p^* \Rightarrow \langle \mbox{ev}_p^*\lambda , \psi \rangle=\lambda(\psi_p)$), then we can use that $\mbox{ker}\,\mbox{ev}_p=(\mbox{img}\,\mbox{ev}_p^*)^{\perp}$.

This way we see that this embedding perfectly suits the interpretation that we made in section (\ref{Initial Approach}) about the zeros of a section being inside a hyperplane.

\begin{lemma}\label{zero correspondence}
The Kodaira correspondence $f$ links the zeros of a section $\psi\in\Gamma(L)$ with the intersection points of $f(p)$ with the hyperplane $[\psi^\perp]$.
\end{lemma}

\begin{proof}
W.l.o.g. let $p\in M$ be any point and $\psi\in\Gamma(L)$ any section.

$$\psi_p=0 \Leftrightarrow \psi\in \mbox{ker}\,\mbox{ev}_p \Leftrightarrow \psi \in (\mbox{img}\,\mbox{ev}_p^*)^{\perp} \Leftrightarrow \mbox{img}\,\mbox{ev}_p^*\stackrel{def}{\tiny =}f(p) \subset [\psi^\perp] \Leftrightarrow f(p)\cap [\psi^\perp]\neq \emptyset$$

Thus a section $\psi$ has a zero at $p$ if and only if $f(p)\in [\psi^\perp]$.
\end{proof}

But that happens in a smooth setting. To connect this in the discrete setting we need to start with the same embedding but only defined in the vertices of the simplicial complex $M$ with vertex set $V$.

$$f:V\rightarrow \mbox{P}(\Gamma(L)^*), \ i\mapsto \mbox{img}\, \mbox{ev}_i^*$$

Next we want to extend this by mapping the edges of $M$ through a simple demand: \emph{the edges of $M$ are mapped to geodesics in $\mathbb{CP}^{n-1}$} where $n$ is the number of vertices in $M$.

By selecting a non-vanishing base section $X$ of the sections $\Gamma(L)$ of a discrete hermitian line bundle (sec: \ref{base section}) we can identify $\Gamma(L)$ with $\mathbb{C}^n$ and thus identify $\mbox{P}(\Gamma(L)^*)$ with $\mathbb{CP}^{n-1}$. This is why we see the Kodaira correspondence as an embedding into $\mathbb{CP}^{n-1}$.

If we had a smooth setting we would work with the Fubini-Study metric as mentioned in (sec: \ref{Fubini-Study metric}), but in the discrete setting it’s analog is the discrete section product defined in (def: \ref{hermitian inner product}) which we will focus on from now on. 

For every vertex $i\in M$ we also define a \emph{piecewise linear hat functions} $x_i$ to carry the value $1$ at the vertex $i$ and to linearly decent to 0 when approaching other vertices. We define the construction of the extension of the discrete Kodaira correspondence in the following lemma.

\begin{lemma}[Geodesic extension]
Let $f$ be the Kodaira correspondence to a simplicial complex $M$ and $L$ be a discrete hermitian line bundle on $M$. If $\eta$ is a discrete connection on $L$ and $i,j$ are neighbouring vertices, we can extend the Kodaira correspondence on the edge $ij$ through

$$f_{ij}:=[x_i\psi^*+x_j\eta_{ji}^*\psi^*] \ , \ \ \text{ for } \ \ \ 0\neq\psi^*\in L_i^*$$

to define a piecewise smooth extension of $f$ such that the image of the edge $ij$ is a geodesic.
\end{lemma}

\begin{proof}
Since $L$ is a hermitian line bundle we know that $\eta_{ij}=r_{ij}$Id for some $r_{ij}\in \mathbb{C}$ with $|r_{ij}|=1$. $f_{ij}(i)=[\phi_i^*],\ f_{ij}(j)=[r_{ji}\phi_j^*]=[\phi_j^*]$. $x_i\phi+x_j\eta_{ji}^*\psi^*$ resembles a straight line with $\phi_i^* \perp \phi_j^*$ according to the discrete section product, which is why its projection is a geodesic.
\end{proof}

Now we want to think about the extension of the discrete Kodaira correspondence $f$ onto the faces $ijk$ where $i,j,k$ are three neighbouring vertices on $M$. But before attempting that we define ourselves the geodesic triangle $\triangle:=\{[x],[y],[z]\}$ as the triangle spanned by geodesics between the three points $[x],[y],[z]$. Thanks to the geodesic extension of the Kodaira correspondence looking at this geodesic triangles now becomes relevant since $f$ embeds faces $ijk$ to geodesic triangles. Let's define the

\begin{definition}[Shape invariante]
If $[x],[y],[z]\in \mathbb{CP}^{n-1}$, then the shape invariant associated to the geodesic triangle $|\triangle|=\{[x],[y],[z]\}$ is defined as:

$$\Upsilon (|\!\bigtriangleup\!|) = 
\mathrm{Re}\left(\frac{\langle\!\langle x,y\rangle\!\rangle_{\mathbb C}\langle\!\langle y,z\rangle\!\rangle_{\mathbb C}\langle\!\langle z,x\rangle\!\rangle_{\mathbb C}}
{\langle\!\langle x,x\rangle\!\rangle_{\mathbb C}\langle\!\langle y,y\rangle\!\rangle_{\mathbb C}\langle\!\langle z,z\rangle\!\rangle_{\mathbb C}}\right).$$
 
\end{definition}

The invariance is evident when noticing how orthogonal transformations are eaten up by the scalar products and scalar multiplications are canceled by the fraction. By the Fubini-study metric the lengths of the edge between $[x]$ and $[y]$ are given by $\mbox{arcos} \left ( \frac{\abs*{\langle x, y \rangle}}{\abs*{x}\abs*{y}} \right )$ as established in (sec: \ref{Fubini-Study metric}). Thus

$$
\left|\frac{\langle\!\langle x,y\rangle\!\rangle_{\mathbb C}\langle\!\langle y,z\rangle\!\rangle_{\mathbb C}\langle\!\langle z,x\rangle\!\rangle_{\mathbb C}}
{\langle\!\langle x,x\rangle\!\rangle_{\mathbb C}\langle\!\langle y,y\rangle\!\rangle_{\mathbb C}\langle\!\langle z,z\rangle\!\rangle_{\mathbb C}}\right|
=\left|\frac{\langle\!\langle x,y\rangle\!\rangle_{\mathbb C}\langle\!\langle y,z\rangle\!\rangle_{\mathbb C}\langle\!\langle z,x\rangle\!\rangle_{\mathbb C}}
{||x||\,||y||\,||y||\,||z||\,||z||\,||x||}\right| =|\cos(a)\cos(b)\cos(c)|$$

and thus there is a suitable $0\leq \omega_\triangle<2\pi$ such that

$$\frac{\langle\!\langle x,y\rangle\!\rangle_{\mathbb C}\langle\!\langle y,z\rangle\!\rangle_{\mathbb C}\langle\!\langle z,x\rangle\!\rangle_{\mathbb C}}
{\langle\!\langle x,x\rangle\!\rangle_{\mathbb C}\langle\!\langle y,y\rangle\!\rangle_{\mathbb C}\langle\!\langle z,z\rangle\!\rangle_{\mathbb C}} = \cos(a)\cos(b)\cos(c)e^{i\omega_\triangle}$$

Like the shape invariant $\Upsilon$ this is still an invariant. We know that $\omega_\triangle$ is related to the K\"ahler form $\sigma$ that we define in sec. (\ref{chap:kaehler section}) by the following formula from \cite{hangan1994geometrical} \cite{Hangan1994}:

\begin{align}\label{eq:omega_triangle}
\omega_\triangle = 2\int_\triangle \omega_K \text{ mod } 2\pi\mathbb{Z}
\end{align}

Note how we cannot specify the value here up to multiples of $2\pi\mathbb{Z}$. This causes the exact same uncertainty that we got with the discrete approach where the $\mbox{ind}_{ijk}$ was only able to pick values from $\{-1,0,1\}$ if we where not given a 2-form $\Omega$ with the line bundle $L$ (lemma: \ref{lemma: index invariance}). If we are given $\Omega$ then we can adjust $\omega_\triangle$ to take values beyond the bounds fo $[0,2\pi)$ for a more precise answer. Formally approaching it this way will bear the fruit of a far more compact expression for the expected sum of indices on each triangle when we will complete the calculations in section (\ref{bringing together}).

Once we embed the vertices and edges of a simplicial complex into a geodesic triangle $\triangle=\{[x],[y],[z]\} \subset \mathbb{CP}^{n-1}$ we can use the formula (\ref{eq:omega_triangle}) to compute the integral of the K"ahler form as

\begin{align}\label{eq:omega_triangle computation}
2\int_\triangle \omega_K&= \omega_\triangle \text{ mod } 2\pi\mathbb{Z}\\ &=
arg\left(\frac{\langle\!\langle x,y\rangle\!\rangle_{\mathbb C}\langle\!\langle y,z\rangle\!\rangle_{\mathbb C}\langle\!\langle z,x\rangle\!\rangle_{\mathbb C}}
{\langle\!\langle x,x\rangle\!\rangle_{\mathbb C}\langle\!\langle y,y\rangle\!\rangle_{\mathbb C}\langle\!\langle z,z\rangle\!\rangle_{\mathbb C}} \right) \text{ mod } 2\pi\mathbb{Z}
\\ &=
\arg(\langle\!\langle x,y\rangle\!\rangle_{\mathbb C}\langle\!\langle y,z\rangle\!\rangle_{\mathbb C}\langle\!\langle z,x\rangle\!\rangle_{\mathbb C}) \text{ mod } 2\pi\mathbb{Z}
\\ &=
(\ \ \arg(\langle\!\langle x,y\rangle\!\rangle_{\mathbb C})+\arg(\langle\!\langle y,z\rangle\!\rangle_{\mathbb C})+\arg(\langle\!\langle z,x\rangle\!\rangle_{\mathbb C})\ \ ) \text{ mod } 2\pi\mathbb{Z}
\end{align}

The modulo $2\pi\mathbb{Z}$ will deny us the exact solution if you are not given the 2-form $\Omega$, but it is no surprise that a more exact solution is perhaps not possible given a discrete setting where the face inside the geodesic triangle was not specified. The shape of the inside of the triangle determines the exact solution but there is no trivial way to extend the Kodaira correspondence to the faces.
Note also that if w.l.o.g. $x\perp y$ that the computation involving the argument will be not determinable. This will not pose a problem later on since even a tiny amount of smoothing will solve this. Visually $x\perp y$ means that the geodesic triangle is not unique in these situations.

However, if we posses a simplicial complex $M$ with a line bundle $L$ together the curvatures $\Omega_{ijk}$ we can make a more reasonable guess about the integral. This will be done in section (\ref{bringing together}).

\section{About the Transformation Formula}

\quad Before we continue to tackle the intersection of hyperplanes problem we need to establish some calculus tools. A basic theorem from calculus goes as follows:

\begin{theorem}[Transformation formula]

If $U \subset \mathbb{R}^n$ open, $\phi:U\rightarrow \mathbb{R}^n$ is injective and smooth and $f:\phi(U)\rightarrow \mathbb{R}$ is integrable by the Lebesgue measure $\lambda$. 

Then $f \circ \phi$ is integrable on $U$ and

$$\int_{\phi(U)}f \ \lambda = \int_U (f\circ\phi)\cdot |\det(d\phi)|\lambda$$

\end{theorem}

What if we drop the requirement for $\phi$ to be injective? Then $p\in\phi(U)$ might have more than one pre-image, thus making the set $U$ carry redundant values for $\phi(U)$. We can however compensate this by defining $\nu_{\phi}:\phi(U)\rightarrow\mathbb{Z} \ , \ p \mapsto \#\phi^{-1}({p})$ and use $\nu_{\phi}$ to \emph{count} the function values multiple times according to their frequency in the pre-images. Figure (\ref{fig:transformation formula}) shows what is meant when covering values multiple times.

\begin{figure}
\centering
\includegraphics[width=6cm]{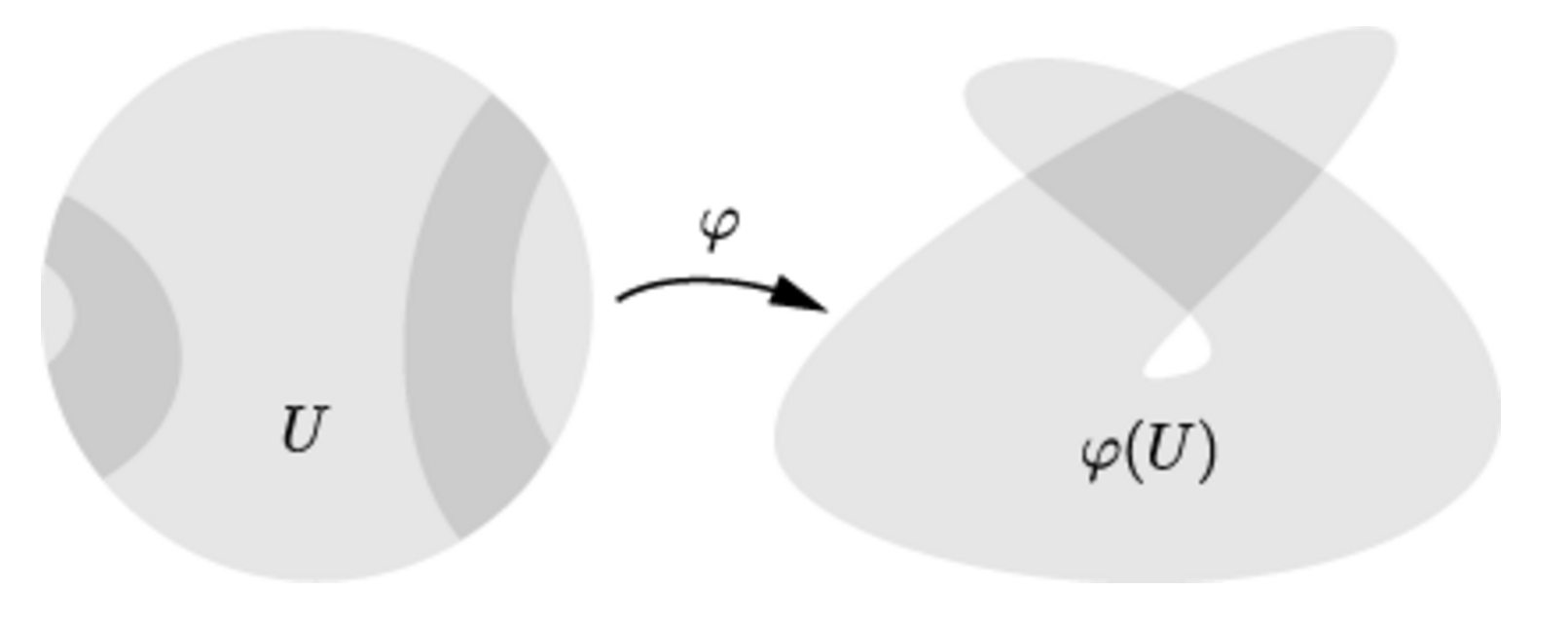}
\caption{Example how $\varphi:U\rightarrow\varphi(U)$ can have multiple preimages that have to be considered. Image from \cite{IntegralGeometry:2015}. }
\label{fig:transformation formula}
\end{figure}

If we then additionally express this in terms of general smooth manifolds we get

\begin{theorem}[Manifold transformation formula]\label{thr:Manifold transformation formula}

Let $M,N$ are compact orientable Riemannian manifolds of equal dimension, $\phi:M\rightarrow N$ smooth and $|\omega_N|$ the density of $N$. If $f:\phi(M)\rightarrow \mathbb{R}$ is $|\omega_N|$-integrable then

$$\int_{\phi(M)}f\cdot \nu_{\phi} \ |\omega_N| = \int_M |\phi^*\omega_N|$$

where $\phi^*$ is the pullback of $\phi$.
\end{theorem}

However, we can modify this formula even more to get rid of the absolute values on both sides. A basic derivation reveals that if $\omega_M, \omega_N$ are the volume forms of $M$ and $N$, then 

$$\phi^*\omega_N=\det(d\phi)\omega_N$$

which is why we need to define

$$\nu_{\phi}^s : \phi(M) \rightarrow \mathbb{Z} \ , \ p \mapsto \sum_{q\in\phi^{-1}({p})} \mbox{sign}(\det(d_q\phi))$$ 

If we remove the absolute value sign on the right hand side of the last transformation formula theorem (\ref{thr:Manifold transformation formula}) we need to compensate the sign change on the left hand side as well, namely by \emph{counting} the overall positive occurrences as well as the negative occurrences of $\det(d\phi)$, precisely by replacing $\nu_{\phi}$ with $\nu_{\phi}^s$.

 We see that $\nu_{\phi}^s=\nu_{\phi}$ if we inserted $|\det(d\phi)|$ instead of $\det(d\phi)$ inside the definition as long as $\det(d_q\phi)\neq0$. The case $\det(d_q\phi)=0$ does not bother us because this occurs only on a null set, thus not altering the outcome of the integral.

By additionally taking $f\equiv1$ we acquire the following theorem that is essential for the next section.

\begin{theorem}[Signed manifold transformation formula]\label{thr: Signed manifold transformation formula}

Let $M,N$ be compact orientable Riemannian manifolds of equal dimension, $\phi:M\rightarrow N$ smooth and $\omega_M,\omega_N$ the respective volume forms of $M,N$. Then

$$\int_{\phi(M)}\nu_{\phi}^s \omega_N = \int_M \phi^*\omega_N = \int_M \det(d\phi)\cdot\omega_M$$

\end{theorem}

\section{Complex Vector Spaces and K\"ahler Angles}\label{chap:kaehler section}

\quad Before we can apply the modified transformation formula we need to explain what K\"ahler angles are and their basic environment needed to define them.

Just as we often identify $\mathbb{C}$ with $\mathbb{R}^2$ we can identify any complex vector space with a real vector space $V$ of a dimension twice as big, but to maintain the multiplicative structure that the imaginary unit $i$ causes in $\mathbb{C}^n$ we need to define the \emph{almost complete structure} $J: V \rightarrow V$ as a linear map with $J^2=-Id$ which represents the multiplication by the imaginary unit $i$. A hermitian vector space is a complex space with an euclidean inner product $\langle .,.\rangle$ such that $J$ is orthogonal. The hermitian inner product can be written as 

$$\langle a,b\rangle_{\mathbb{C}}=\langle a,b\rangle + i\underbrace{\langle Ja,b\rangle}_{\sigma(a,b):=}, \ \forall a,b\in V$$

And we call $\sigma$ the \emph{K\"ahler form} of $V$. Note that if $a\in V\rightarrow a\perp Ja$ since $\langle Ja,a\rangle=\langle JJa,Ja\rangle=\langle -a,Ja\rangle=-\langle Ja,a\rangle$ with the euclidean product.

\begin{definition}[K\"ahler angle]
Let $P$ be a 2-dimensional real plane in a complex vector space with K\"ahler form $\sigma$. The K\"ahler angle is defined as $\eta_P\in[0,\pi]$ such that $\cos\,\eta_P=\sigma(X,Y)$ for any oriented orthonormal basis $X,Y$ of $P$.

\end{definition}

The choice of oriented orthonormal basis is arbitrary since $\sigma$ is invariant under orientation preserving rotations in the plane P. Let us see this using $\theta\in[0,2\pi)$

\begin{align*} 
\sigma ( \cos\theta X -\sin\theta Y, \sin\theta X + \cos\theta Y) & = \cos\theta \sin\theta \underbrace{\langle JX,X\rangle}_{=0}\\
& \ \ + \cos^2\theta\langle JX,Y\rangle \\[2ex]
& \ \ - \sin\theta \cos\theta\underbrace{\langle JY,Y\rangle}_{=0}\\
& \ \ -\sin^2\theta\underbrace{\langle JY,X\rangle}_{= -\langle JX,Y\rangle }\\
\sigma(X,Y) &=\langle JX,Y\rangle
\end{align*}

We say that a 2-dimensional plane $P$ is a \emph{complex line} if $\eta_P\in\{0,\pi\}$ ($\Rightarrow JP = P$) and \emph{real plane} if $\eta_P=\dfrac{\pi}{2}$ ($\Rightarrow JP\perp P$). The real planes are those in the image of the inclusion map from $\mathbb{R}^n\rightarrow \mathbb{C}^n$. So the K\"ahler angle measures the degree of divergence from being a purely complex line.

We will let $P$ be a real 2-dimensional plane in V, which is a hermitian vector space by inheritance. If we now take a positively oriented orthonormal basis $X,Y \in P$ then we can define a \emph{plane bound complex structure} $J_P$ by:

$$J_P:P\rightarrow P, \ \ J_PX=Y \ , \ J_PY=-X$$

\begin{theorem}\label{thr:plane_bound}
Let $X,Y$ be a positively oriented basis of a 2-dimensional real plane P in a complex vector space V. Let $\pi_P:V\rightarrow P$ be the orthogonal projection on $P$ and $J|_P$ the restriction of $J$ to $P$. Then the plane bound complex structure $J_P$ fulfils:
$$\pi_P\circ J|_P=\cos\,\eta_P \cdot J_P$$
\end{theorem}

\begin{proof}

Since $Z\in P$ write $Z=\langle X,Z\rangle X + \langle Y,Z\rangle Y$.


\begin{align*} 
\pi_P\circ J|_P(Z) &=\pi_P\circ J (\langle X,Z\rangle X + \langle Y,Z\rangle Y \\
&=\langle X,Z\rangle \pi_P(JX) + \langle Y,Z\rangle \pi_P(JY) \\
&=\langle X,Z\rangle \left(\underbrace{\langle X,JX}_{=0}\rangle X + \langle Y,JX\rangle Y\right) + \langle Y,Z\rangle \left(\langle X,JY\rangle X + \underbrace{\langle Y,JY\rangle}_{=0} Y \right) \\
&=\langle X,Z\rangle \langle Y,JX\rangle Y + \langle Y,Z\rangle \langle X,JY\rangle X \\
&=\cos\eta_P\langle Y,Z\rangle \underbrace{-X}_{=J_PY}+\cos\,\eta_P\langle X,Z\rangle \underbrace{Y}_{=J_PX}\\
&=\cos\,\eta_PJ_P(\langle X,Z\rangle X + \langle Y,Z\rangle Y ) = \cos\,\eta_PJ_P(Z)
\end{align*}

So if $Z\in P$ then $\pi_P\circ J|_P(Z)=\cos\,\eta_PJ_P(Z)$. 
\end{proof}

Figure (\ref{fig:kaehler angle}) shows how this theorem (\ref{thr:plane_bound}) can be understood.

\begin{figure}
\centering
\includegraphics[width=5cm]{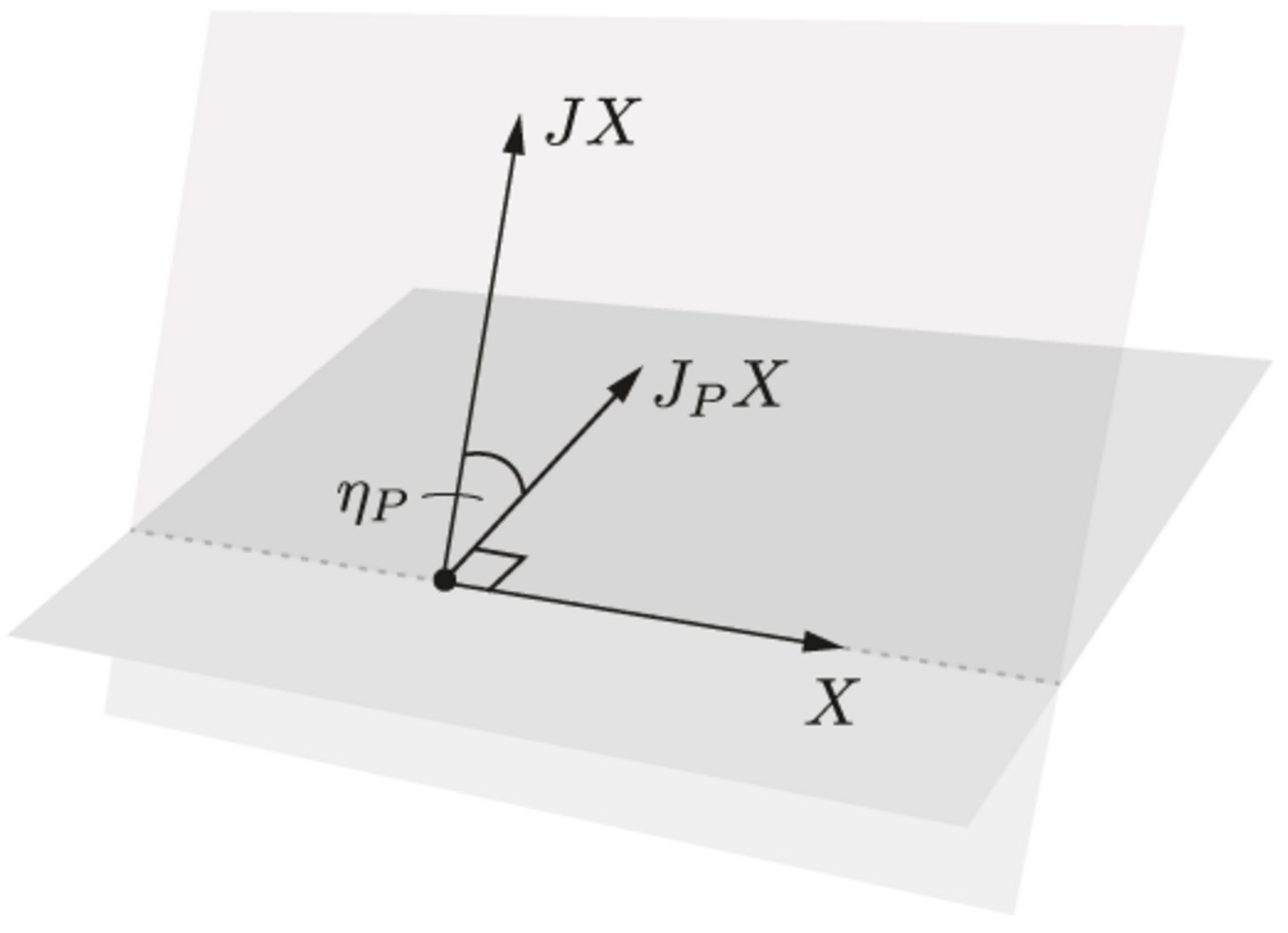}
\caption{Example how $X$ is mapped by $J$ and $J_P$. Image from \cite{IntegralGeometry:2015} }
\label{fig:kaehler angle}
\end{figure}

If $Q$ is a subspace of $V$ we define $\pi_Q$ as the orthogonal projection on $Q$. With this we can set up the following

\begin{lemma}
Let $X,Y$ be a positively oriented basis of a 2-dimensional real plane P in a complex vector space V. If $\eta_P\notin \{0,\pi\}$ then $\exists$ 2-dimensional real subspace $Q \perp P$ such that $P\oplus Q=P+JP$ and a orientation preserving orthogonal matrix $A$ with $\pi_QJ|_Q=\sin\,\eta_PA$ and

$$J|_{P+Q}= \left( \begin{array}{cc} \cos\eta_P \ J_P & -\sin\eta_P \ A^* \\ \sin\eta_P \ A &-\cos\eta_P \ J_Q \end{array} \right)  $$
\end{lemma}

\begin{proof}
If $\eta_P\notin\{0,\pi\}$ then $JP\cap P\neq P $ and $JP+P$ is 4 dimensional real linear space with in inherited Euclidean product. So we can find an orientation preserving orthogonal basis $X,Y,Z,W$ of $JP+P$ such that $X,Y\in P$ and $Z,W\in P^{\perp}$. $Q:=span(Z,W)$. Then if $b\in P$ with unit length we get 

$$1=|Jb|^2=|\pi_P(Jb)|^2+|\pi_Q(Jb)|^2=\cos^2\eta_P+\underbrace{|\pi_Q(Jb)|^2}_{\Rightarrow =\sin^2\eta_P}$$

Thus $\pi_Q\circ J|_P=\sin\,\eta_PA$ for some orthogonal map $A:P\rightarrow Q$ with $J_Q=A\circ J_P\circ A^*$ since $J$ is orthogonal. Thus $J_{P+Q}$ is an endomorphism and can be written in matrix form as such that for $Z\in P\oplus Q \Rightarrow$

$$J|_{P+Q}(Z)=  \left( \begin{array}{cc} \cos\eta_PJ_P & -\sin\eta_PA^* \\ \sin\eta_PA &-\cos\eta_PJ_Q \end{array} \right) \cdot \left( \begin{array}{c} \pi_P(Z)\\ \pi_Q(Z) \end{array} \right)$$

\end{proof}

Hyperplanes in $\mathbb{CP}^n$ can be oriented and therefore we can define the signed intersection of hyperplanes with surfaces. The canonical orientation of $\mathbb{C}^{n+1}$ can be given by taking any basis $z_1,...,z_{n+1}$ and defining the orientation of the real $\mathbb{R}^{2n+2}$ space by the n-form $z_1\wedge Jz_1\wedge ...\wedge z_{n+1}\wedge Jz_{n+1}$.

Let, $M$ be an 2-dimensional oriented surface. $f:M\rightarrow \mathbb{CP}^n$. For $p\in M$ and $X,Y\in \mbox{T}_pM$  positive oriented we let $d_pf(X),d_pf(Y)$ span the two dimensional real plane inside $\mbox{T}_p\mathbb{CP}^n$. We now extend $d_pf(X)\wedge d_pf(Y)$ by $v_1\wedge ...\wedge v_{2n-2}$ with a positively oriented basis $\{v_1,….,v_{2n-2}\}$of the hyperplane  $h=span(v_1,…,v_{2n-2})$ that intersects $f(p)$ to determine the sign of the intersection. If $span(d_pf(X),d_pf(Y) \cap h \neq \emptyset$ It can only acquire the values $\pm 1$, and $0$ if there is no intersection. If $d_pf(x),d_pf(x)$ lie inside a hyperplane then the sign is not defined. However, for our purpose this will be of no concern since the set of hyperplanes that do that form a null set and take no effect in the later integral. We denote the sign of the intersection by $\#_{\pm}(h\cap f(p))$ and $\#_{\pm}(h\cap f(M))=\sum_{p\in M}\#_{\pm}(h\cap f(p))$.

And now we can see that the orientation of the intersection depends on the K\"ahler angle. To see that 

$$d_pf(X),d_pf(Y),v_1, ...,v_{2n-2}$$\label{coseta_P sign}

is positive oriented we need to know if $\langle(Jd_pf(X),d_pf(Y)\rangle=\cos\, \eta_P$ has a positive sign or not for $P=span_{\mathbb{R}}(d_pf(X),d_pf(Y))$ because if $\cos\,\eta_P>0$, then $Jd_pf(X),d_pf(Y)$ point in the similar direction as $\{z,Jz\}$ would in the canonical orientation that we defined above. If $\cos\,\eta_P<0$ we have a negative orientation more like $\{z,-Jz\}$.

\section{Average Intersection with Hyperplanes}\label{Average intersection with hyperplanes}

\quad Now the times has come to unite the section about transformation formulas together with the knowledge on K\"ahler angles. This section will dedicate itself exclusively to the proof of a crucial theorem needed to find an expression for the expected sum of indices on a face. We separate this from the rest due to it being the integral geometric core of this thesis.

\begin{theorem}
Let $f:M\rightarrow \mathbb{CP}^k$ be an embedding of a oriented Riemannian surface. $2\leq k\in\mathbb{N}$. Let $H_f:=\{\text{hyperplanes } h\subset \mathbb{CP}^k \  \text{ that intersect } \ \mbox{img}(f)\}$ and $dA$ be the surface area element of the surface. Then

$$\int_{h\in H_f}\#_{\pm}(h\cap f(M)) dK=\dfrac{1}{2\pi}\int_M\cos\,\eta(f(p)) dA$$

where $\eta(f(p))$ is the K\"ahler angle of the real plane spanned by $d_pf(X),d_pf(Y)$ for a positve oriented orthogonal basis $X,Y$ of $\mbox{T}_pM$.
\end{theorem}

\begin{proof}
In a nutshell this is a special case of the signed manifold transformation formula of theorem (\ref{thr: Signed manifold transformation formula}). To see this we need to find a parametrization $\phi:E\rightarrow H$ of some set $E$ such that the image $\mbox{img}(\phi|_E)$ is the set of hyperplanes that intersect the surface. Remember that when we defined the signed intersection that we mentioned that there could be a problem if $d_pf(X),d_pf(Y)$ are inside a hyperplane? This won't be a problem here since the set of hyperplanes that may do so form a null set and can be ignored due to the integral later. 

Remember that we can identify any hyperplane in $\mathbb{CP}^k$ with $\mathbb{CP}^k$ itself.  A suitable set for the domain of $\phi$ is $E=f^*\mbox{P}(\mbox{T}\mathbb{CP}^k)$ which is the pull back bundle of $\mathbb{CP}^k$ by $f$ that has been projectivised by the natural $\mathbb{C}$-projectivation $\mbox{P}$. If $U\subset M$ we write

$$E_U:=\{ \ \{p\}\times \mbox{P}(T_{f(p)}\mathbb{CP}^k) \ : p\in U\}, \ \ E:=E_M, \ \ E_p:= \{p\}\times \mbox{P}(T_{f(p)}\mathbb{CP}^k)$$

We use this set to define the function $\phi : E\rightarrow H$ as:

$$\text{for } 0\neq\xi\in f(p)\in \mathbb{CP}^k \ \ \psi\in\mathbb{C}^{k+1}/f(p) \ \ \phi(p,[\phi]):=[|\xi|^2\psi-\langle\xi , \psi\rangle_{\mathbb{C}}\xi ]$$

where we compute $\langle\xi , \psi\rangle_{\mathbb{C}}$ in homogeneous coordinates which works fine since we can easily verify that the full expression is well defined. The definition of $\phi$ aims to be suitable for the signed manifold transformation formula.

Lets make sure that $\phi(E)=H_f$.

``$\subseteq$`` If $\tau \in \phi(E) \Rightarrow \exists p\in M, \xi\in f(p), \psi\in\mathbb{C}^{n+1}/f(p) : \tau=[\hat{\tau}]=[|\xi|^2\psi-\langle\xi , \psi\rangle_{\mathbb{C}}\xi] \Rightarrow$. To see that $f(M)\cap \tau\neq\emptyset$ we can w.l.o.g. take $\xi\in f(p)$ again and compute the scalar product

$$\langle |\xi|^2\psi-\langle\xi , \psi\rangle_{\mathbb{C}}\xi , \xi \rangle = |\xi|^2\langle\psi,\xi\rangle-\langle\xi , \psi\rangle_{\mathbb{C}}\langle\xi , \xi \rangle = 0$$

``$\supseteq$`` If we are given $h\in H_f \Rightarrow \exists p\in M: h\cap f(p)\neq\emptyset$. So if $\xi\in f(p)$ and $[\hat{h}]^\perp=h$ then $\langle\hat{h},\xi\rangle_{\mathbb{C}}=0$. W.l.o.g. pick $|\xi|=1$ and $\psi=\hat{h}\in\mathbb{C}^{k+1}$ and

$$[\hat{h}]=[1^2\hat{h}-\overbrace{\langle \xi,\hat{h}\rangle_{\mathbb{C}}}^{=0}\xi]=|\xi|^2\psi-\langle\xi ,\psi\rangle_{\mathbb{C}}\xi \ \ \ \Rightarrow \ \ \ h\in \mbox{img}(\phi|_E)$$

thus $h$ is in $T_{f(p)}\mathbb{CP}^n$ and this concludes why $\phi(E)=H$. Our goal would be to be able to apply theorem (\ref{thr: Signed manifold transformation formula}) to get the following chain of equalities:

\begin{equation}\label{eq:dream equality}
\int_{h\in H|_f}\#_{\pm}(h\cap f(M)) dK\stackrel{!}{=}\int_{\phi(E)}\nu_{\phi}^s\ \omega_{\mathbb{CP}^k}=\int_E\phi^*\omega_{\mathbb{CP}^k}=\int_M\int_{E_p}\det(d_p\phi)dA\wedge\omega_{\mathbb{CP}^{k-1}}
\end{equation}

But we can only show ``$\stackrel{!}{=}$`` after further computations, namely the computation of $\det(d\phi)$.

Let $\epsilon >0,\ a\in \mathbb{R}$. The differential of a curve $\gamma : (a-\epsilon,a+\epsilon) \rightarrow \mathbb{CP}^k$ at $a$ has to be independant of any scalar multiplication $\lambda:(a-\epsilon,a+\epsilon)\rightarrow \mathbb{C}$ if $\gamma=[\lambda\hat{\gamma}]$. So 

$$\gamma' \text{ has to be equivalent to } [\lambda'\gamma + \lambda\gamma'].$$

This is why we see differentials as an elements of $\mbox{Hom}_{\mathbb{C}}(\gamma(a),\mathbb{C}^{k+1}/\gamma(a))$ which we then identify with $\mbox{T}_a\mathbb{CP}^k$. We denote the differential of $\gamma$ at $a$ by $\dot{\gamma}(a)$.

In order to differentiate $\phi$ we need to write $\phi$ into a form that is more suitable for differentiating by using coordinates in a local neighbourhood $U \subseteq M$ around a fixed point $q$ and a lift into $\mathbb{C}^{k+1}$. Define $H_p:=\{x\in \mathbb{CP}^k \ : \ f(p) \perp x\}=f(p)^{\perp}$ and let $\xi:U\rightarrow \psi+H_q$. In the following the point $q$ is fixed and $p$ is the variable.

$$\hat{\phi}:U\times H_p \rightarrow \mathbb{C}^{n+1}, \ \ \ \hat{\phi}(p,\psi):=|\xi(p)|^2\psi-\langle\xi(p) , \psi\rangle_{\mathbb{C}}\xi(p) $$

With $\xi(p) \in f(p).\ |\xi(q)|=1$. We constructed it this way to earn the commutation of $\phi\circ\pi_U=[\hat{\phi}]$.

$$\begin{array}{ccc} U\times H & \xrightarrow[\hspace{3em}]{\hat{\phi}} & \mathbb{C}^{n+1} \\ \pi_{U} \xdownarrow{3ex}   & & \xdownarrow{3ex} [.]\\ E|_U & \xrightarrow[\hspace{3em}]{\phi} & \mathbb{CP}^n \end{array}$$

Differentiating $\hat{\phi}$ with a curve $(p(t),\psi(t))$ such that $p(0)=q$,  $\psi(0)=\psi$ and indicating this with a dot $\dot{p},\dot{\psi}$ will lead to

\begin{align*} 
d\hat{\phi}(\dot{p},\dot{\psi}) &=d(\overbrace{\langle\xi(p),\xi(p)\rangle}^{=|\xi|^2}\psi)-d(\langle\xi(p) , \psi\rangle_{\mathbb{C}}\xi(q))\\
&=2\underbrace{\langle d\xi(\dot{p}),\xi(q)\rangle}_{=0}\psi+\underbrace{\langle \xi(q),\xi(q)\rangle}_{=1}\dot{\psi}\\
&\ \ - \langle d\xi(\dot{p}) , \psi\rangle_{\mathbb{C}}\xi(q) - \underbrace{\langle \xi(q) , \dot{\psi}\rangle_{\mathbb{C}}}_{=0}\xi(q)\\
&\ \ - \underbrace{\langle \xi(q) , \psi\rangle_{\mathbb{C}}}_{=0}d\xi(\dot{p})\\
&= \dot{\psi} - \langle d\xi(\dot{p}) , \psi\rangle_{\mathbb{C}}\xi(q)
\end{align*}

where the underbraces are evident due to $\psi,\dot{\psi},d\xi(\dot{p})\in\xi(q)^{\perp}$. Hence the linear map $d\hat{\phi}:\mbox{T}_qU \oplus H_q \rightarrow H_q^{\perp}\oplus H_q$ must have the form

$$\Bigl( \begin{array}{cc} -\langle d\xi(\dot{p}) , \psi\rangle_{\mathbb{C}} & 0 \\ 0 & id_{H_q} \end{array} \Bigr), \ \Rightarrow \ \det d\hat\varphi = \det \langle d\xi, \psi\rangle_{\mathbb C}$$

Thanks to \cite{IntegralGeometry:2015} we can also make a statement on the following isometry:

\begin{displayquote}
\say{Since the Fubini-Study metric on $\mathrm{T}_q \mathbb C\mathrm P^k = \mathrm{Hom}_{\mathbb C} \bigl(f(q), \mathbb C^{k+1}/f(q)\bigr)= \mathrm{Hom}_{\mathbb C} (f(q), H)$ is given by the Frobenius norm and $|\phi_0|=1$, it coincides with the metric on H. Thus if furthermore $|\psi|=1$ the same argument shows that the maps in the following diagram are isometric:}
\end{displayquote}

$$\begin{array}{ccc} \mbox{T}_qM\oplus H_q & \xrightarrow[\hspace{3em}]{d\hat{\phi}} & \mbox{T}_{\hat{\phi}(q,\psi)}\mathbb{C}^{k+1} \\ \pi_{\mbox{T}_qM} \xdownarrow{3ex}   & & \xdownarrow{3ex} [.]\\ E_{(q,\phi)} & \xrightarrow[\hspace{3em}]{\phi} & \mbox{T}_{(\phi(q,[\phi])}\mathbb{CP}^k \end{array}$$

Thus we conclude that
$$\left.\varphi^\ast \omega_{\mathbb C\mathrm P^k}\right|_{(q,[\psi])} = \det \langle d_q\xi, \psi\rangle_{\mathbb C}\, \left.dA \wedge \omega_{\mathbb C\mathrm P^{k-1}}\right|_{(q,[\psi])}$$

Now we can apply theorem (\ref{thr:plane_bound}) for $P=\mathrm{img}\, d_q\xi$ to compute $\det \langle d_q\xi, \psi\rangle_{\mathbb C}$. First note that

\begin{align*}
\langle J_P d_q\xi ,J \psi\rangle &= \langle J_P d_q\xi ,\pi_P J(\pi_P(\psi)+\pi_Q(\psi))\rangle \\ & = \cos\eta_P\, \langle J_P d_q\xi ,J_P\pi_P(\psi)\rangle - \sin\eta_P\,\langle J_P d_q\xi ,A^\ast\pi_Q(\psi)\rangle\\ & = \cos\eta_P\, \langle d_q\xi ,\psi\rangle - \sin\eta_P\,\langle A\, J_P d_q\xi ,\psi\rangle
\end{align*}

and consequently if we transform $\psi\mapsto J\psi$ we get.

$$\langle J_P d_q\xi ,\psi\rangle = -\cos\eta_P\, \langle d_q\xi ,J\psi\rangle + \sin\eta_P\,\langle A\, J_P d_q\xi ,J\psi\rangle$$

To compute the determinant we take any orthogonal basis and derive the determinant from its representative matrix. Thus, for some $X=\dot{p}\in \mbox{T}_p \mathrm M$ such that $|X|=1$, we obtain the basis $\{X,JX\}$ and compute

\begin{align*}
\det \langle d_q\xi, \psi\rangle_{\mathbb C} &= \langle J_P d\xi(X),\psi\rangle \langle d\xi(X),J\psi\rangle - \langle d\xi(X),\psi\rangle\langle J_P d\xi(X),J\psi\rangle\\
&= (-\cos\eta_P\, \langle d\xi(X) ,J\psi\rangle + \sin\eta_P\,\langle A\, J_P d\xi(X) ,J\psi\rangle\bigr)\langle d\xi(X),J\psi\rangle\\ & \ \ \ - \langle d\xi(X),\psi\rangle \bigl(\cos\eta_P\, \langle d\xi(X) ,\psi\rangle - \sin\eta_P\,\langle A\, J_P d\xi(X) ,\psi\rangle\bigr)\\ &= -\cos\eta_P\, \bigl(\langle d\xi(X),\psi\rangle^2 + \langle d\xi(X),J\psi\rangle^2)\\ & \ \ \ + \sin\eta_P\,\bigl(\langle\psi,A\, J_P d\xi(X)\rangle\langle d\xi(X),\psi\rangle + \langle J\psi,A\, J_P d\xi(X)\rangle\langle d\xi(X),J\psi\rangle\bigr).
\end{align*}

Let $\pi_Q, \pi_{[\psi]}$ denote the orthogonal projections to Q and the complex subspace $[\psi]$. Since $[\phi]=span_{\mathbb{R}}(\phi,J\phi)$ we see that

$$\langle d\xi(X),\psi\rangle^2 + \langle d\xi(X),J\psi\rangle^2 = \left|\langle d\xi(X),\psi\rangle\dfrac{\phi}{||\phi||} + \langle d\xi(X),J\psi\rangle\dfrac{J\phi}{||J\phi||}\right|^2 = |\pi_{[\phi]}( d\xi(X))|^2$$

If we replace $\sin\, \eta_P A$ with $\pi_QJ$ then

\begin{align*}
\det \langle d_q\xi, \psi\rangle_{\mathbb C} &= -\cos\eta_P\, |\psi|^2|\pi_{[\psi]}( d\xi(X))|^2\\& \ \ \  + \langle\psi,\pi_Q J J_P d\xi(X)\rangle\langle d\xi(X),\psi\rangle + \langle J\psi,\pi_Q J J_P d\xi(X)\rangle\langle d\xi(X),J\psi\rangle
\end{align*}

Notice how $\langle \psi,\pi_Q J J_P d\xi(X)\rangle\langle d\xi(X),\psi\rangle\equiv0$ everywhere since in any situation, $\psi\in P,Q,(P\oplus Q)^{\perp}$, at least one term in the product will vanish.

So finally we can come back to the original  computation we have been trying to make for a long time. We divide the integral over the pullback bundle $E$ up into each base point $p\in M$ and its respective tangent space $E_p$.

\begin{align*}
\int_{\mathrm E} \phi^\ast\omega_{\mathbb C\mathrm P^k} &= \int_{\mathrm M} \Bigl(\int_{\mathrm E_p} \det\langle d_p\xi,\psi\rangle_{\mathbb C}\, \omega_{\mathbb C\mathrm P^{k-1}}\Bigr)\, dA\\
&= \int_{\mathrm M} \cos\eta \underbrace{\Bigl(\int_{\mathrm E_p} |\psi|^2|\pi_{[\psi]}( d\xi(X))|^2  \omega_{\mathbb C\mathrm P^{k-1}}\Bigr)}_{c:=} dA\\
&= c\int_{\mathrm M} \, \cos\eta\, dA
\end{align*}

where c is a constant coming from the fact that the integral

$$\int_{\mathrm S^{n-1}}|\psi|^2|\pi_{[\psi]}(\varrho)|^2$$

is independent of the choice of $\varrho\in \mathrm S^{n-1}$. Now we finally have the computations visible to easily see that if $p\in M , \ \ h\in \phi(E_p) \Rightarrow$ the signed intersection $\#_{\pm}(h\cap f(p)) = \nu_{\phi}(p)$ almost everywhere. See section (\ref{coseta_P sign}, page: \pageref{coseta_P sign}) to verify how the sign of $\cos\,\eta_P$ affects the sign of the intersection to finally allow us to proof the ``$\stackrel{!}{=}$`` equality of equation (\ref{eq:dream equality}). 

$$\int_{h\in H|_f}\#_{\pm}(h\cap f(M)) dK=\int_{\phi(E)}\nu_{\phi}^s\ \omega_{\mathbb{CP}^k}$$

The idea behind it is the following: if the hyperplane $h$ cuts $f(p)$ in $i\in\mathbb{N}$ points, then $E$ will cover the same hyperplanes $i$ times. Additionally, the sign of a single intersection coincides as seen here: 

$$\#_\pm(h\cap f(p)=\mbox{sign}(\cos\,\eta(f(p))=\mbox{sign}(\det(d_p\hat{\phi}))$$ 
Now, since the kinematic measure is invariant under rotations, $dK$ is a constant multiple of $\omega_{\mathbb C\mathrm P^n}$ and there is a universal constant C such that

$$\int_{\{h\mid h\cap f(M) \neq \emptyset\}} \#_\pm(h\cap f(M))\,dK = C\int_{\mathrm M} \, \cos\eta(f(p))\, dA$$

As C is universal we can compute it by a specific example: Let $f\colon \mathbb C\mathrm P^1\to \mathbb C\mathrm P^n$ be some projective line. Then the K\"ahler angle is 0 everywhere. Since any two distinct projective hyperplanes have exactly one positive intersection point we have

\begin{align*}
1 = \int_{h} dK &= \int_{\{h\mid h\cap f(\mathbb{CP}^1) \neq \emptyset\}} \#_\pm(h\cap f(\mathbb{CP}^1))\,dK\\ &= C\int_{\mathbb C\mathrm P^1} \, \underbrace{\cos\eta(f(p))}_{=1}\, dA\\ &= \mathrm{Area}(\mathbb C\mathrm P^1)\cdot C\\ &= 2\pi\, C \ \ \Rightarrow C= \tfrac{1}{2\pi}
\end{align*}

Which concludes the proof of the theorem.

\end{proof}

This concludes the tool for the integral geometric approach. The application will be handled in section (\ref{bringing together}).

\chapter{Solution and Discussion}\label{chap:Solution and Discussion}

\quad Now that we have worked so much on tools and theorems to work with we will have to put them all together in this section to finally express the average sum of indices. 

\section{Bringing Everything Together}\label{bringing together}

\begin{definition}
Let $\triangle$ represent the triangle spanned by the 3 neighboring vertices $i,j,k\in M$ where $M$ is once again a closed simplicial complex in $\mathbb{R}^3$. Let $L$ be a discrete hermitian line bundle on $M$ and $n\in\mathbb{N}$ be the number of vertices. $f:\triangle\rightarrow\mathbb{CP}^{n-1}$ be an embedding, then for a section $\Gamma(L)$ we define its index

$$\hat{I}_{ijk}(f):= \int_{\{h^\perp \mid h\in\mathbb{CP}^{n-1}\}} \#_\pm(h^\perp\cap f(\triangle))\,dK$$

and thus the index density as

$$\hat{Z}_{ijk}(f):= \dfrac{\hat{I}_{ijk}(f)}{\mbox{Area}(ijk)}$$

\end{definition}

If $f$ is the Kodaira correspondence we interpret $\hat{I}_{ijk}(f)$ as the sum of indices. The justification for the interpretation of this definition was established in lemma (\ref{zero correspondence}).




    
We now still want to derive a useful way of writing $\hat{I}_{ijk}$ using all of the above results. 


\begin{align}\label{big equality}
\hat{I}_{ijk}(f) &=\int_{\{[\psi]^\perp \mid \psi\in\mathbb{C}^{n}\}} \#_\pm(h\cap f(\triangle))\,dK\\
&=\int_{\{h\mid h\cap f(\triangle) \neq \emptyset\}} \#_\pm(h\cap f(\triangle))\,dK \\
&=\tfrac{1}{2\pi} \int_{\triangle} \, \cos\eta(f(p))\, dA \ \ m^2\\
&=\tfrac{1}{2\pi}  \int_{\triangle } \, f^*\sigma\, dA \\
&=\tfrac{1}{2\pi}(\tfrac{1}{2}\omega_{f\circ\triangle} + \pi l)\\
\end{align}

Lets look at what we did for each equation step by step respectively:

\begin{enumerate}
\item We applied the definition of $\hat{I}_{ijk}(f)$.
\item Remove unnecessary hyperplanes from the integral since $\#_\pm(\emptyset)=0$.
\item This was the subject of the theorem in (\ref{Average intersection with hyperplanes}) and replacing $k$ by $n-1$.
\item Definition of K\"ahler form.
\item Here we used the equation (\ref{eq:omega_triangle}) and defined $l\in\mathbb{Z}$ accordingly. $\omega_{f\circ\triangle}$ is the embedding into a geodesic triangle in $\mathbb{CP}^{n-1}$. From now on we will omit the $f$ there and just just write $\omega_{f\circ\triangle}=\omega_{\triangle}$.
\end{enumerate}

But due to the mod $2\pi\mathbb{Z}$ issue we have to be very careful with the choice of $l$. This arises from equation (\ref{eq:omega_triangle}) in section (\ref{Embedding into}). We have to look further to see what $l$ must be.

However, we have not yet even applied the smoothing operator $S_t$ in these equations. Without applying the smoothing the result is fairly boring and $\omega_\triangle$ won't be determinable by equation (\ref{eq:omega_triangle computation}). We can use the above setting and modify it a little to compute the zeros of the smoothed sections.

Remember from lemma (\ref{zero correspondence}) that the zeros of a section $\phi\in\Gamma(L)$ can be identified with the intersection points of the Kodaira correspondence embedding of the surface $M$ into $\mathbb{CP}^{n-1}$ and the hyperplane $[\phi]^\perp$. For any $t\geq0$ this means that for the smoothed section $S_t\phi$ we know (homogeneous coordinates wherever necessary): 

$$p\in \mbox{ker}\,S_t\phi\Leftrightarrow \langle f(p),S_t\phi\rangle=0\Leftrightarrow \langle S_tf(p),\phi\rangle=0$$

And with the discrete Kodaira correspondence on each vertex $i\in m$ we have $f(i)=[\delta_i]$. Remember that the main theorem from (\ref{Average intersection with hyperplanes}) did not require $f$ to be the Kodaira correspondence, which is why we can replace $f\mapsto [S_tf]$  and still use that theorem and all of the equalities in (\ref{big equality}). Thus:

\begin{align*}
& \ \ \ \arg(\langle\!\langle f(i),f(j)\rangle\!\rangle_{\mathbb C}\langle\!\langle f(j),f(k)\rangle\!\rangle_{\mathbb C}\langle\!\langle f(k),f(i)\rangle\!\rangle_{\mathbb C})
\\ &= \arg(\langle\!\langle S_t\delta_i,S_t\delta_j\rangle\!\rangle_{\mathbb C}\langle\!\langle S_t\delta_j,S_t\delta_k\rangle\!\rangle_{\mathbb C}\langle\!\langle S_t\delta_k,S_t\delta_i\rangle\!\rangle_{\mathbb C}) \stackrel{!}{=} \omega_\triangle^t
\end{align*}

Where we just defined $\omega_\triangle^t$ in analogy to $\omega_\triangle$ but for the geodesic triangle by the vertices $[S_t(i)],[S_t(j)],[S_t(k)]$. It would be nice to use the equation (\ref{eq:omega_triangle computation}) and just continue the big equality (\ref{big equality}) by replacing $\omega_\triangle^t$ with the above equality. However, due to the mod $2\pi$ problem we have to guess $l\in\mathbb{Z}$ from the equations (\ref{big equality}). Also for $t=0$ we can't determine $\omega_\triangle^0$ as $\arg(0)$ is not trivial. We can shift $\omega_\triangle^t$ to be in $(-\pi,\pi)$, and igrnore the rare equalities with $\pm\pi$ as they form a null set.
We perform the choice of $l\in\mathbb{Z}$ by taking into consideration the result of the discrete approach for $t=0$ in equation (\ref{eq:unsmoothed_result}). The results in both approaches have to be equal and thus we deduce that for $t=0$

$$\hat{I}_{ijk}(S_tf)=\dfrac{\Omega_{ijk}}{2\pi} =\tfrac{1}{2\pi}(\tfrac{1}{2}\omega_\triangle^0 + \pi l) \ \ \ \ \Rightarrow \ \ \  \dfrac{\Omega_{ijk}}{\pi}=\underbrace{\dfrac{1}{2\pi}\omega_\triangle^0}_{-\tfrac{1}{2}<\ \ \ <\tfrac{1}{2}\ } + \underbrace{l}_{\in\mathbb{Z}}$$

$$l=\mbox{round}\left( \dfrac{\Omega_{ijk}}{\pi} \right) \ \ \ \Rightarrow \omega_\triangle^0:= 2(\Omega_{ijk}-\pi l)\ \ \ $$

Where $\mbox{round}:\mathbb{R}\rightarrow \mathbb{Z}$ is the rounding function to the nearest integer. This is only true for $t=0$. We expect that the zeros of sections move continuously on the surface when smoothing the section since $S_t$ is continous. From this we deduce that for small $t$ the chosen value of $l$ will remain the same and we will have to compute $\omega_\triangle^t$ with the above $\arg(\dots)$ expression for $t>0$. This is why we continue the equalities from (\ref{big equality}) as

\begin{align*}\label{smoothed equality}
\hat{I}_{ijk}(S_tf) &=\tfrac{1}{4\pi}\omega_\triangle^t + \dfrac{1}{2}\mbox{round} \left( \dfrac{\Omega_{ijk}}{\pi}\right )\\
\end{align*}

Notice that if we are not given a curvature 2-form $\Omega$ with the geometry, then the connection $\eta$ that we add our self as defined after definition (\ref{discrete hermitian line bundle with curvature}) results in $\Omega_{ijk}\in(-\pi,\pi)$, meaning that $l=0$. 

Next we have to address another problem for increasing values of $t$. $\omega_\triangle^t$ is still trapped inside $(-\pi,\pi)$. If $t$ becomes big enough, the argument of the calculation of the $\omega_\triangle^t$ expression will jump above or below of $\pm\pi$, but this information will be lost in the computation. To fix this we need to define that for $t>0$ that $\omega_\triangle^t$ is not bound by $(-\pi,\pi)$, and that it recognizes if the increase of $t$ just caused the argument to move above $\pi$, and if so, continue without loosing that progress. In other words, we want $\omega_\triangle^t$ to be a continuous function.

Let us look at an example to clear this thought up. If $0<t_1\approx t_2,\ t_1<t_2$ and $\omega_\triangle^{t_1}\approx\pi$ and $\omega_\triangle^{t_2}\approx -\pi$, then the value of $\omega_\triangle^{t_2}$ should be replaced by $\omega_\triangle^{t_2}+2\pi$.

In our discrete computations, the only way to implement this is to increase $t$ bit by bit and compare each new value with his predecessor to determine what should have been.

All in all put together we finally reach the expression we have aimed to get for so long and will express this in the following theorem:

\begin{theorem}[Average number of zeros of smoothed sections]\label{thr:main}

Let $M$ be a closed simplicial complex with $n\in \mathbb{N}$ vertices with a discrete hermitian line bundle $L$ with a discrete curvature $\Omega$ and $t\geq0$. Let $\triangle$ be the triangle spanned by neighboring vertices $i,j,k\in M$ and $f$ be the Kodaira correspondence

Let $\omega_{\triangle}^t$ be a smooth function of $t>0$ such that 

$$\omega_{\triangle}^t=\arg(\langle\!\langle S_t\delta_i,S_t\delta_j\rangle\!\rangle_{\mathbb C}\langle\!\langle S_t\delta_j,S_t\delta_k\rangle\!\rangle_{\mathbb C}\langle\!\langle S_t\delta_k,S_t\delta_i\rangle\!\rangle_{\mathbb C}) \text{ mod } 2\pi$$

and $\omega_\triangle^0= 2\left(\Omega_{ijk}-\mbox{round}\left( \dfrac{\Omega_{ijk}}{\pi} \right)\pi  \right)$.

Then the average number of signed zeros of random sections smoothed by $S_t$ on $\triangle$ $ijk$ is

$$
\hat{I}_{ijk}(S_tf)=\tfrac{1}{4\pi}\omega_\triangle^t + \dfrac{1}{2}\mbox{round} \left( \dfrac{\Omega_{ijk}}{\pi}\right )
$$

The area density is

$$\hat{Z}_{ijk}(S_tf)= \dfrac{\hat{I}_{ijk}(S_tf)}{\mbox{Area}(ijk)}$$

\end{theorem}

\section{Results and Visualisation}

\quad Now that we have managed to establish a way to computing the density of indices of random sections of discrete hermitian line bundles on simplicial complexes we can finally briefly discuss the outcome. We will not even try to apply the discrete approach from chapter (\ref{chap:Discrete Index Approach}) as we did not manage to simplfy the integral expression (\ref{eq:probability density}). Instead we will focus only on the final expression as in theorem (\ref{thr:main}).


Lets define for neighboring vertices $ijk$ of a closed simplicial complex $M$ with $n\in\mathbb{N}$ verticies in total:

$$P_{ijk}(t):=\dfrac{\tfrac{1}{4\pi}\omega_\triangle^t + \dfrac{1}{2}\mbox{round} \left( \dfrac{\Omega_{ijk}}{\pi}\right )}{\mbox{Area}(ijk)}$$

Lets visually inspect the results inside jReality\footnote{\url{http://www3.math.tu-berlin.de/jreality/}}, a visualisation software created inside the Technical University of Berlin. The code written simply computes $P_{ijk}(t)$ on each surface for a fixed $t\geq0$ and colorized the minimal value blue, the maximal value red, and all the values in between with a rainbow color gradient.

The base of the program was written by Felix Kn\"oppel in pursuit of the the original hunch of the non-trivial distribution of zeros as noticed in the \emph{Global Optimal Direction Field} paper\cite{Knoppel:2013:GOD}. In Felix's code smoothed sample sections where counted through as seen in figure (\ref{fig:oldschool}). The size of the points determine the quantity and the color the value of the of the index sum. 

\begin{figure}
\hfill
\includegraphics[height=6cm]{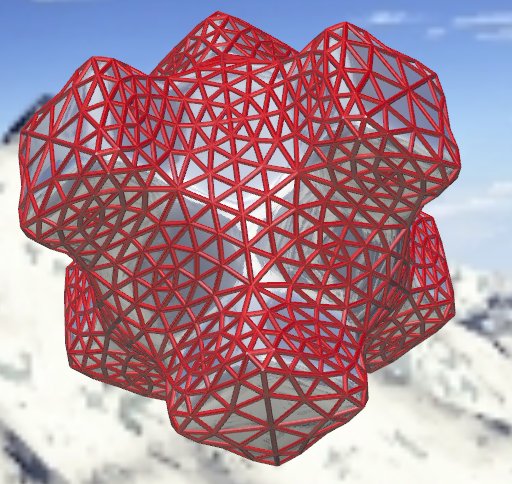}
\hfill
\includegraphics[height=6cm]{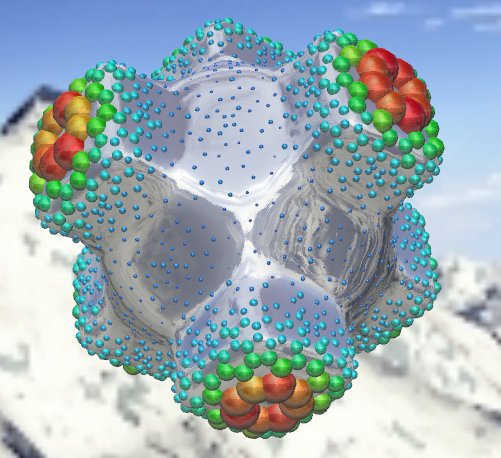}
\hfill\hspace{0.1cm}
\caption{1000 smoothed sample sections and their accumulated indices counted on each face of the closed \emph{schwarz patch} surface}
\label{fig:oldschool}
\end{figure}

To ease the computation of $P_{ijk}(t)$ we use theorem (\ref{smoothing convergence}) to justify that eigenvectors to higher eigenvalues become more irrelevant with increasing $t$ and are therefore excluded from the computation. In other words, we choose $1\leq k\leq n \in \mathbb{N}$ and for any section $\phi\in\Gamma(L)$ we project it orthogonally onto $span(\varphi_1,...,\varphi_k)$ using $\pi_k$ where $\varphi_i$ are the eigenvectors of $S_t$ and $\Delta$ ordered by the size of their respective eigenvalues. In a nutshell we cut off eigenvectors of higher eigenvalues:

$$\phi=\sum_{i=1}^n\mu_i\varphi_i\Rightarrow\pi_k(\phi)=\sum_{i=1}^k\mu_i\varphi_i$$

This allows faster computations of $\omega_\triangle^t$ with similar precision. For grater values of $t$ we can decrease the value of $k$ as long as $k\geq \dim(Eig(\lambda_1))$ so that the smoothing convergence determines that the same limit will be met. In formal words:

$$\text{if } k\geq \dim(Eig(\lambda_1)) \Rightarrow \lim_{t\rightarrow\infty} [S_t\phi]=\lim_{t\rightarrow\infty} [S_t\circ \pi_k(\phi)]$$

This becomes evident when observing the prove of theorem (\ref{smoothing convergence}) in section (\ref{sec:Smoothing Operator}).

Let us observe the effects on an ellipse in figure (\ref{fig:newschool}) from a slightly tilted view. Above the geometry we see its $P_{ijk}(t)$ value histogram distribution from its minimum to its maximum. In the first picture we still see the effects caused by cutting down $\phi$ to $\pi_k(\phi)$ as $t=0$. With increasing $t$ the positive zeros become more likely at the pointy ends of the ellipsoid while most of the surface has index zero. By Poincare-Hopf, the index sum of an ellipse is equal to $\chi($ ellipse $)=2$, which lets us guess that the ellipsoid will acquire only one positive zeros at each of the pointiest tips. This figure was loaded from a random simplicial simplex in the shape of an ellipsoid, meaning that we had to compute $\Omega$ ourselfs and thus have $\Omega_{ijk}\in(-\pi,\pi)$.

\begin{figure}
\centering
\includegraphics[width=7cm]{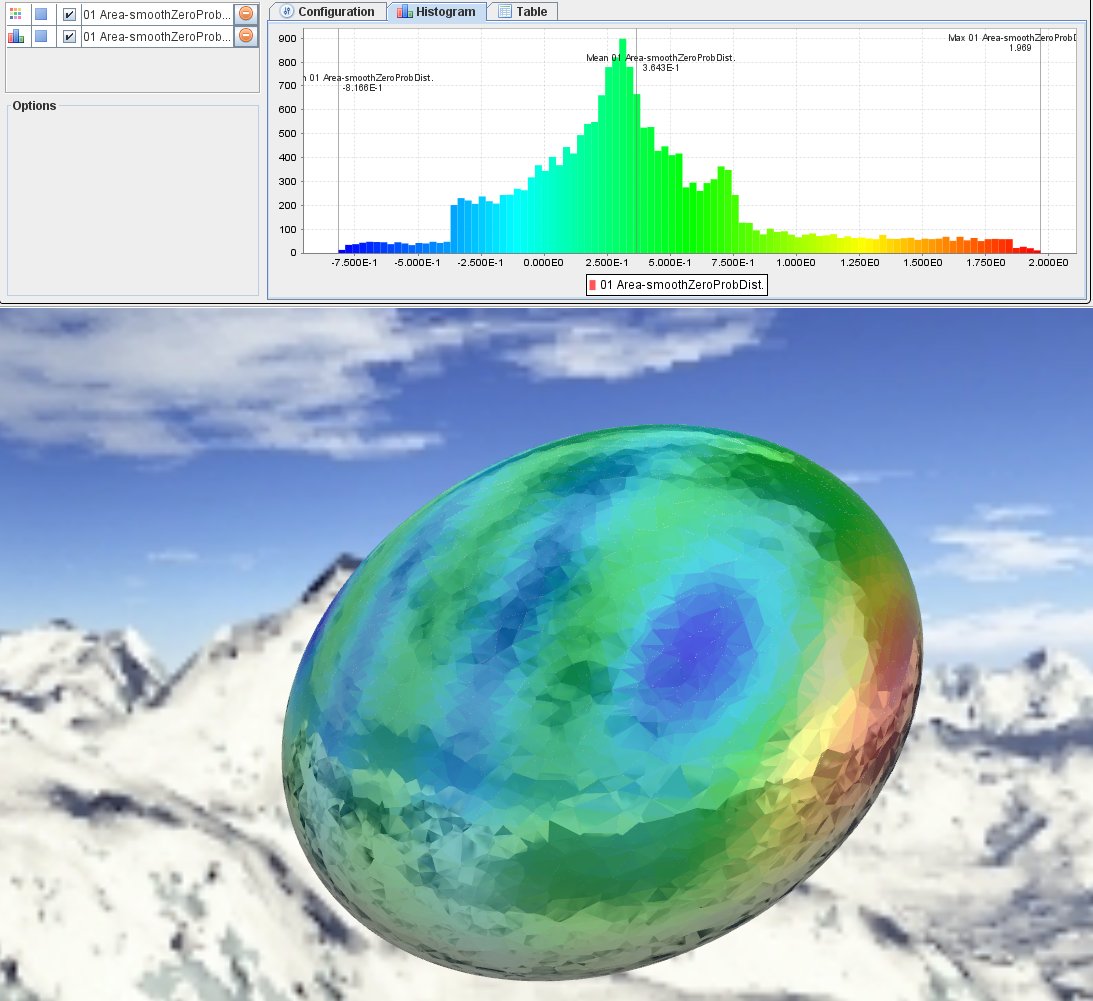}
\includegraphics[width=7cm]{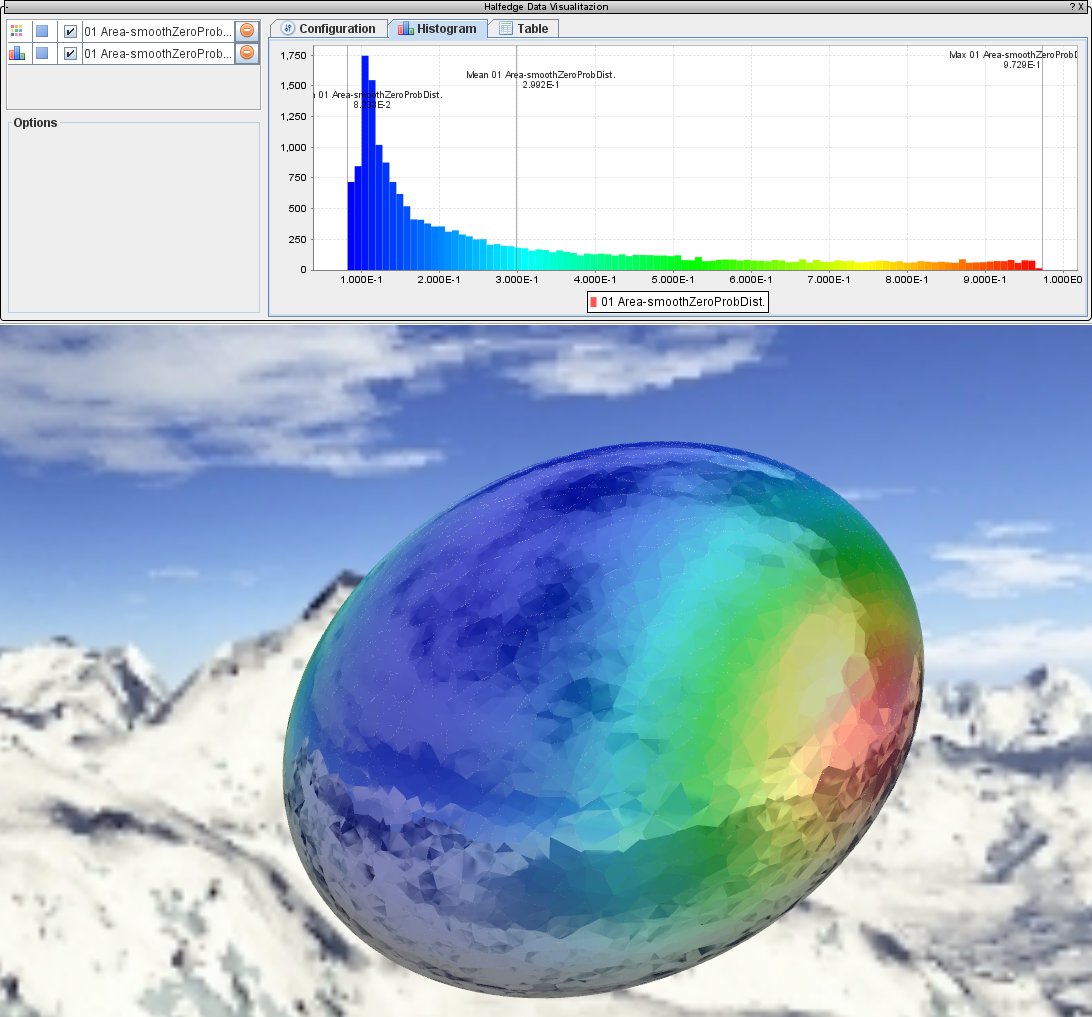}\\
\includegraphics[width=7cm]{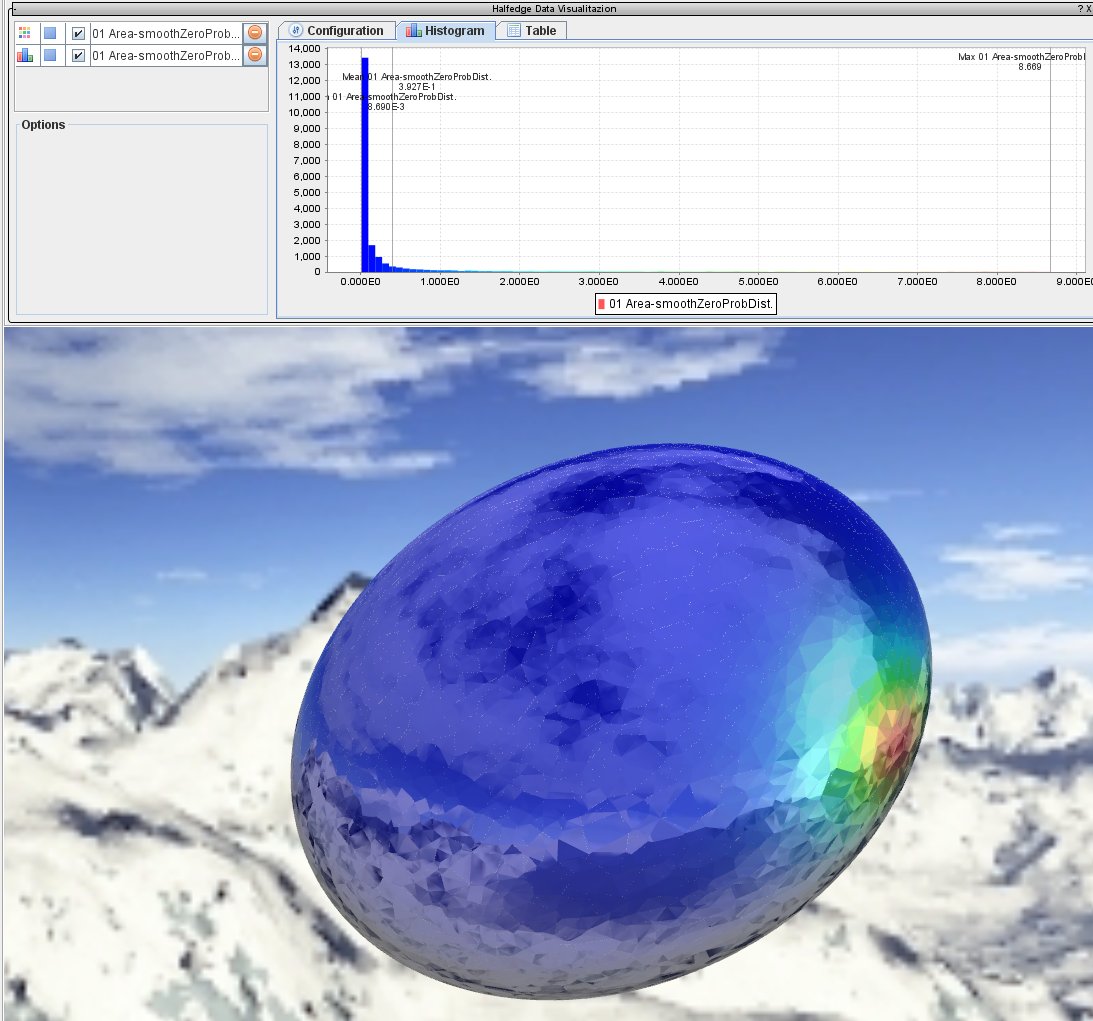}
\includegraphics[width=7cm]{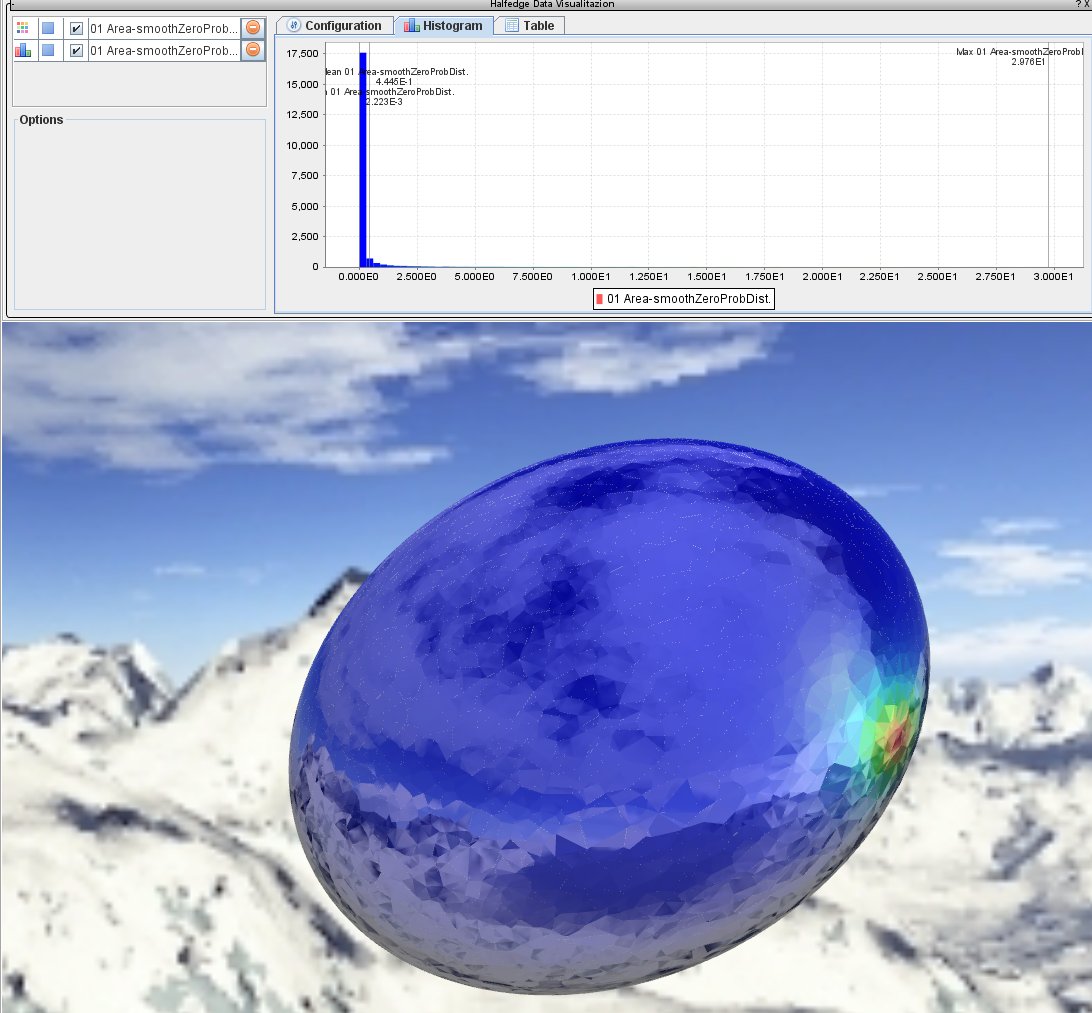} 
\caption{ellipse with 10k vertices. From left to right, top to bottom the value of $t$ was increased starting at $0$.}
\label{fig:newschool}
\end{figure}

\newpage

\bibliographystyle{amsalpha}
\bibliography{Ref_Marcel_Padilla_Bachelor_Thesis}

\end{document}